\documentclass{amsart}
\usepackage{amsthm,latexsym,amsfonts,amsmath,amssymb,mathrsfs,array,manfnt,stmaryrd}
\usepackage[pdftex]{graphicx}
\usepackage[all]{xy}
\usepackage{paralist}
\usepackage{framed}
\usepackage{cancel}
\usepackage{enumitem}
\usepackage{asymptote}
\usepackage{units}
\usepackage{setspace}
\usepackage{amscd}
\usepackage{microtype}
\usepackage{tikz}
\usetikzlibrary{matrix}
\usetikzlibrary{calc}
\usepackage{tikz-cd}

\newcommand{\into}{\hookrightarrow}

\newcommand{\abs}[1]{\left\lvert#1\right\rvert}

\newcommand{\Z}{\ensuremath{\mathbb{Z}}}
\newcommand{\R}{\ensuremath{\mathbb{R}}}
\newcommand{\C}{\ensuremath{\mathbb{C}}}

\newcommand{\M}{\mathcal{M}}
\newcommand{\Mbar}{\overline{\mathcal{M}}}
\newcommand{\Hbar}{\overline{\mathcal{H}}}

\renewcommand{\P}{\ensuremath{\mathbb{P}}}

\newcommand{\Hcomb}{\mathcal{H}_{\mathrm{prof}}}

\newcommand{\Hbarcomb}{\overline{\mathcal{H}}_{\mathrm{prof}}}
\newcommand{\Htropcomb}{\mathcal{H}_{\mathrm{prof}}^{\trop}}
\newcommand{\Htrop}{\mathcal{H}^{\trop}}
\newcommand{\Htilde}{\widetilde{\mathcal{H}}}
\newcommand{\Htildecomb}{\widetilde{\mathcal{H}}_{\mathrm{prof}}}

\DeclareMathOperator{\lcm}{lcm}

\DeclareMathOperator{\val}{val}

\DeclareMathOperator{\MSC}{MSC}

\DeclareMathOperator{\len}{len}
\DeclareMathOperator{\stratum}{X}
\DeclareMathOperator{\cone}{Cone}

\usepackage{color}

\newcommand{\margincolor}{red}      
\definecolor{darkgreen}{rgb}{0,0.7,0}

\newcounter{margincounter}
\newcommand{\marginnum}{
\ifnum\value{margincounter}<10
\textcolor{\margincolor}{\begin{picture}(0,0)\put(2.2,2.4){\circle{9}}\end{picture}\footnotesize\arabic{margincounter}}
\else\ifnum\value{margincounter}<100
\textcolor{\margincolor}{\begin{picture}(0,0)\put(4.256,2.5){\circle{11}}\end{picture}\footnotesize\arabic{margincounter}}
\else
\textcolor{\margincolor}{\begin{picture}(0,0)\put(6.8,2.5){\circle{14}}\end{picture}\footnotesize\arabic{margincounter}}
\fi\fi
}

\newcommand{\bP}{\mathbf{P}}

\newcommand{\CC}{\mathbb{C}}

\newcommand{\PP}{\mathbb{P}}

\newcommand{\Qbar}{\overline{\mathbb{Q}}}

\newcommand{\cH}{\mathcal{H}}

\newcommand{\cM}{\mathcal{M}}

\newcommand{\cT}{\mathcal{T}}

\newcommand{\bB}{\mathbf{B}}

\newcommand{\st}{\thickspace|\thickspace}

\theoremstyle{plain}
\newtheorem{theorem}{Theorem}
\numberwithin{theorem}{section}

\newtheorem*{thm*}{Theorem}

\newtheorem{prop}[theorem]{Proposition}

\newtheorem{cor}[theorem]{Corollary}

\newtheorem{lem}[theorem]{Lemma}

\theoremstyle{definition}

\newtheorem{Def}[theorem]{Definition}

\theoremstyle{remark}

\newtheorem{rem}[theorem]{Remark}

\usepackage{tikz}
\usetikzlibrary{matrix}
\usepackage{fullpage}
\usepackage{bbm}

\DeclareMathOperator{\trop}{trop}

\usepackage{hyperref}

\DeclareMathOperator{\edges}{Edges}

\DeclareMathOperator{\TP}{TP}

\DeclareMathOperator{\TLT}{TLT}
\DeclareMathOperator{\TLTtilde}{\widetilde{TLT}}
\DeclareMathOperator{\Tbar}{\overline{\mathcal{T}}}
\DeclareMathOperator{\lambdathurst}{\lambda_{\mathrm{Thurst}}}
\DeclareMathOperator{\lambdascal}{\lambda_{\mathrm{scal}}}
\DeclareMathOperator{\lambdadom}{\lambda_{\mathrm{dom}}}
\DeclareMathOperator{\CCom}{CC}
\DeclareMathOperator{\ConeCC}{ConeCC}
\DeclareMathOperator{\stab}{St}
\DeclareMathOperator{\edgedeg}{EdgeDeg}
\DeclareMathOperator{\legdeg}{LegDeg}

\DeclareMathOperator{\slope}{Slope}
\DeclareMathOperator{\Mod}{PMod}
\DeclareMathOperator{\CP}{\mathbb{C}\mathbb{P}^1}
\DeclareMathOperator{\tree}{\tau}
\DeclareMathOperator{\lcmdeg}{LcmDeg}

\begin{document}

\title{Thurston obstructions and tropical geometry}
\author{Rohini Ramadas}
\email{rohini.ramadas@warwick.ac.uk}
\address{Warwick Mathematics Institute, University of Warwick, Coventry, UK}
\date{\today}

\begin{abstract}
We describe an application of tropical moduli spaces to complex dynamics. A post-critically finite branched covering $\varphi$ of $S^2$ induces a pullback map on the Teichm\"uller space of complex structures of $S^2$; this descends to an algebraic correspondence on the moduli space of point-configurations of $\P^1$. We make a case for studying the action of the tropical moduli space correspondence by making explicit the connections between objects that have come up in one guise in tropical geometry and in another guise in complex dynamics. For example, a Thurston obstruction for $\varphi$ corresponds to a ray that is fixed by the tropical moduli space correspondence, and scaled by a factor $\ge 1$. This article is intended to be accessible to algebraic and tropical geometers as well as to complex dynamicists. 
\end{abstract}
\maketitle

\section{Introduction}

Topologically speaking, every rational function $f:\CP\to\CP$ is an orientation-preserving branched covering of the 2-sphere. In 1982, Thurston addressed the question: which branched coverings are rational functions? In order for this to be a well-posed question, we restrict to the case of branched coverings that are \textit{post-critically finite}, i.e. whose postcritical set is finite. Given an orientation-preserving branched cover $\varphi:S^2\to S^2$ with finite post-critical set $\bP$, Thurston showed that ``pullback of complex structures along $\varphi$" is a well-defined holomorphic self-map $\TP_{\varphi}$ of the Teichm\"uller space $\cT(S^2,\bP)$ of complex structures on $S^2\setminus \bP$. Essentially tautologically, a fixed point of $\TP_{\varphi}$ is an identification $S^2\cong \CC\PP^1$ under which $\varphi$ is isotopic, relative to $\bP$, to a PCF rational function on $\CC\PP^1$. Furthermore, $\TP_{\varphi}$ is distance-nonincreasing with respect to the Teichm\"uller metric on $\cT(S^2,\bP)$, and, under a mild additional hypothesis\footnote{\label{fn:hyperbolicorbifold}The hypothesis is that $\varphi$ has \textit{hyperbolic orbifold}, see \cite{DouadyHubbard1993} for a definition and complete classification of excluded $\varphi$. The only examples that are excluded are certain examples with four or fewer post-critical points.} on $\varphi$, $\TP_{\varphi}^2$ is strictly distance-decreasing. In the latter case, to which we restrict in this article, either (A) $\TP_{\varphi}$ has fixed point, which is necessarily unique, or (B) the branched covering $\varphi$ carries a topological object called a \textit{Thurston obstruction}. (This proof was written up by Douady and Hubbard in \cite{DouadyHubbard1993}.) 

A Thurston obstruction is a multicurve --- a collection of pairwise-disjoint simple closed curves on $S^2\setminus \bP$ --- that has a type of invariance called \textit{$\varphi$-stability}, and that satisfies another condition having to do with the spectral properties of its \textit{Thurston linear transformation} (defined in Section \ref{sec:multicurves}). In fact, Pilgrim \cite{pilgrim2001canonical} showed that if $\varphi$ is not isotopic to a PCF rational function, then it has a \textit{canonical} Thurston obstruction admitting an informal interpretation as follows: for all $x\in\cT(S^2,\bP)$, as $r\to\infty$, $\TP_{\varphi}^{r}(x)$ ``goes to infinity", i.e. leaves every compact subset of $\cT(S^2,\bP)$. As $\TP_{\varphi}^{r}(x)$ goes to infinity, the complex structure on the sphere degenerates. As the complex structure degenerates, there is at least one simple closed geodesic curve whose length goes to zero. The set of all such curves is the canonical obstruction. 

Thurston's theorem (as well as the theory developed in its proof) is still central in modern $1$-variable complex and arithmetic dynamics. For example, one consequence of the theorem is that (with exceptions\footnote{The exceptions are \textit{flexible Latt\`es maps}: A flexible Latt\`es map is a rational map on $\CC\PP^1$ that is descended from an affine map on an elliptic curve $E$, via the hyperelliptic map $E\to \CC\PP^1$. Flexible Latt\`es maps have square degree, and are PCF with exactly four post-critical points.}) PCF maps are ``rigid", i.e. do not deform in non-isotrivial families, and are defined over $\Qbar$. Teichm\"uller space $\cT(S^2,\bP)$ is a nonalgebraic complex manifold; it is the universal cover of the moduli space $\cM_{0,P}$, an algebraic variety parameterizing point-configurations on $\CC\PP^1$. In 2013, Koch \cite{Koch2013} showed that $\TP_{\varphi}$ descends to an algebraic multivalued self-map $\MSC_{\varphi}$ of $\cM_{0,\bP}$ known as the \textit{moduli space correspondence}. Fixed points of $\MSC_{\varphi}$ are in bijection with PCF rational functions that are \textit{Hurwitz equivalent} to $\varphi$ (Defintion \ref{def:hurwitzeq}), but that are not necessarily isotopic to $\varphi$. 

There is a \textit{moduli space of tropical curves} $\M_{0,\bP}^{\trop}$, which is a polyhedral object obtained by gluing convex cones together along faces (\cite{Mikhalkin2006, BilleraHolmesVogtmann}, see Section \ref{sec:tropicalcurve}). Very informally speaking, $\M_{0,\bP}^{\trop}$ can be thought of as a space parameterizing ``paths to infinity" in $\M_{0,\bP}$ \cite{AbramovichCaporasoPayne2015}. It follows directly from the existence and properties of the \textit{tropical Hurwitz space} constructed in \cite{CavalieriMarkwigRanganathan2016} that the moduli space correspondence induces a \textit{tropical moduli space correspondence} $\MSC_{\varphi}^{\trop}$, a piecewise linear multivalued action on $\M_{0,\bP}^{\trop}$. 

In this article, we synthesize concepts well-understood in the field of rational dynamics and Teichm\"uller theory with concepts well-understood in the field of tropical curves and their moduli spaces. We believe that the resulting synthesis results in conceptual clarity and valuable perspective. We first (Section \ref{sec:tropicalteichmuller}) put forward a candidate for the tropicalization of Teichm\"uller space. There is currently no precise definition of what the term tropicalization might mean as applied to Teichm\"uller space, which is not an algebraic variety --- in Section \ref{sec:nonalgebraictropical} we discuss the various meanings of the term tropicalization, and justify our candidate for tropical Teichm\"uller space. In Section \ref{sec:tropicalpullback} we interpret the collection of Thurston linear transformations as together comprising a \textit{tropical pullback map} $\TP^{\trop}$. We then show:

\begin{theorem}[Informal summary of Section \ref{sec:tropicalcorrespondence}]\label{thm:main}\hfill
    \begin{enumerate}
        \item (Proposition \ref{prop:tropicalcorrespondencefrompullback} and Diagram \ref{diag:tropical}.) $\MSC^{\trop}_{\varphi}$ descends from $\TP^{\trop}_{\varphi}$. 
        \item (Proposition \ref{prop:localtropicalcorrespondence}.) Every locally-defined single-valued linear branch of $\MSC_{\varphi}^{\trop}$ has the same matrix as some Thurston linear transformation associated to $\varphi$.
        \item (Proposition \ref{prop:weaklyfixedray} and Corollary \ref{cor:obstructionsandrays}.) $\MSC_{\varphi}^{\trop}$-fixed rays correspond to non-empty multicurves on $S^2$ that are $\varphi'$-stable for some $\varphi'$ Hurwitz equivalent to $\varphi$. $\MSC_{\varphi}^{\trop}$-fixed rays that are scaled by a factor $\ge 1$ correspond to multicurves that are Thurston obstructions for some $\varphi'$ Hurwitz equivalent to $\varphi$. \label{it:fixedray}
    \end{enumerate}
\end{theorem}

Theorem \ref{thm:main} provides evidence that the tropical moduli space correspondence captures important aspects of the dynamics near infinity of the algebraic moduli space correspondence. It is therefore of interest to study the global dynamics of the tropical correspondence. However, the technical tools needed to undertake this study are currently under-developed. We explain some of these difficulties in Remarks \ref{rem:truetropicalcorrespondence} and \ref{rem:notdefinedonlink}. 

\begin{rem}\label{rem:link}
 $\M_{0,\bP}^{\trop}$ is a cone, that is, it is a union of rays that all meet at their common origin, which is the cone point of $\M_{0,\bP}^{\trop}$. The link 
$\Delta_{\bP}$ of $\M_{0,\bP}^{\trop}$ is a compact polyhedral complex whose points correspond to the rays of $\M_{0,\bP}^{\trop}$. The topology of $\Delta_{\bP}$ can be explicitly related to the topology of $\M_{0,\bP}$ \cite{ChanGalatiusPayne2016}. It is tempting to try to obtain from $\MSC_{\varphi}^{\trop}$ an induced action on $\Delta_{\bP}$, and exploit the compactness of $\Delta_{\bP}$ in order to study that action. However, this is not straightforward, as we explain in Remark \ref{rem:notdefinedonlink}. 
\end{rem}
 
\subsection{Organization}
In Section \ref{sec:M0n} we give background on the moduli space $\M_{0,\bP}$ of point-configurations, its compactification $\Mbar_{0,\bP}$ and its tropicalization $\M_{0,\bP}^{\trop}$. Section \ref{sec:teich} mirrors Section \ref{sec:M0n}, giving background on Teichm\"uller space, a partial compactification, and a (non-standard) candidate for its tropicalization. In Section \ref{sec:ThurtonsTheorem} we give background on Thurston's theorem, introducing the pullback map, $\varphi$-stable multicurves, Thurston linear transformations and Thurston obstructions. In Section \ref{sec:modulicorrespondence} we give background on Hurwitz spaces and the moduli space correspondence. In Section \ref{sec:HurwitzCompactificationandTropicalization} we give background on compactifications and tropicalizations of Hurwitz spaces. Section \ref{sec:tropicalcorrespondence} contains the main content of this article --- we give a definition of the tropical moduli space correspondence and prove Theorem \ref{thm:main}. The reader will see that the proofs are relatively straightforward once the set-up is appropriately made. 



\subsection{Acknowledgements} I am grateful for useful conversations with Xavier Buff, Charles Favre, Sarah Koch, Curtis McMullen, Kevin Pilgrim, Dylan Thurston and Rob Silversmith, and grateful to Kevin Pilgrim and to an anonymous referee for useful comments on previous drafts. This work was partially funded by UKRI EPSRC New Investigator Award EP/X026612/1.

\section{The moduli space of point-configurations}\label{sec:M0n}

We introduce the moduli space of point-configurations, its compactification, and its tropicalization. Throughout this article, we will assume that $\bP$ is a finite set of cardinality at least $3$. The moduli $\M_{0,\bP}$ is a smooth quasiprojective $(\abs{\bP}-3)$-dimensional variety parameterizing injective maps $\iota:\bP\to \CP$, up to post-composition by automorphisms of $\CP$. If $\bP_1\subset \bP_2 \subset S^2$, there is a forgetful map $\mu:\M_{0,\bP_2}\to\M_{0,\bP_1}$. 

\subsection{Stable curves and the Deligne-Mumford compactification.}\label{sec:M0nbar}

A \textit{prestable} $\bP$-marked genus-$0$ curve is a pair $(C,\iota)$, where $C$ is an at worst nodal curve of arithmetic genus $0$ (i.e. a tree of $\CP$s attached at simple nodes), and $\iota$ is an injective map from $\bP$ into the smooth locus of $C$. In this article, all prestable curves are assumed to be genus-$0$. A prestable curve is called stable if each irreducible component has at least $3$ points that are either nodes or marked. The dual graph of a nodal curve is the graph which has a vertex for each irreducible component of the curve, and an edge between two vertices if the corresponding irreducible components of the curve meet at a node. Given a prestable $\bP$-marked curve, its marked dual tree $\tree$ is the pair $(\underline{\tree}, \widehat\iota)$, where $\underline{\tree}$ is the dual graph of $C$ (which is necessarily a tree), and $\widehat\iota$ is the map from $\bP$ to the set of vertices of $\underline{\tree}$ sending $p\in\bP$ to the vertex of $\underline{\tree}$ corresponding to the irreducible component of $C$ that is marked by $p$. We define the valence of a vertex $v$ of $\tree$ to be 
$$\val(v):=(\mbox{number of edges adjacent to $v$})+\abs{\widehat\iota^{-1}(v)}.$$

A $\bP$-marked tree $\tree$ is called stable if every vertex has valence at least $3$. A prestable $\bP$-marked curve is stable if and only if its marked dual tree is stable. The Deligne-Mumford compactification $\Mbar_{0,\bP}$ of $\M_{0,\bP}$ parameterizes, up to isomorphism, stable $\bP$-marked curves $(C,\iota)$.

Given a stable $\bP$-marked tree $\tree$, the locus in $\Mbar_{0,\bP}$ of curves with marked dual tree isomorphic to $\tree$ is a locally closed set $\stratum_{\tree}^{\circ}$ whose closure is an irreducible smooth subvariety $\stratum_{\tree}$ called a \textit{boundary stratum} of $\Mbar_{0,n}$. The codimension of $\stratum_{\tree}$ is equal to the number of edges of $\tree$. The containment $\stratum_{\tree_1}\subset \stratum_{\tree_2}$ holds if and only if $\tree_2$ can be obtained from $\tree_1$ by contracting a subset of edges of $\tree_1$.

A prestable $\bP$-marked curve $(C,\iota)$ admits a unique \textit{stabilization}, which is a nodal curve $C'$ together with a surjective degree-$1$ map $\stab:C\to C'$ with the property that $(C',\stab\circ\iota)$ is a stable $\bP$-marked curve. The map $\stab$ may contract some irreducible components of $C$ to points. Given $\bP_1\subset \bP_2$, the forgetful map $\mu$ extends to a map $\Mbar_{0,\bP_2}\to \Mbar_{0,\bP_1}$, sending $(C,\iota:\bP_2\into C)$ to the stabilization of $(C,\iota|_{\bP_1})$. The $\bP_1$-marked dual tree of $\mu((C,\iota))$ is determined by the $\bP_2$-marked dual tree of $(C,\iota)$ --- we will denote by $\mu_*$ the induced map from the set of (isomorphism classes of) $\bP_2$-marked stable trees to the set of (isomorphism classes of) $\bP_1$-marked stable trees.

\subsection{Tropical curves and the tropical moduli space}\label{sec:tropicalcurve} The topological space of a $\bP$-marked tree $\tree$ is the topological space obtained from the usual underlying topological space of $\underline{\tree}$ by attaching, for every $p\in \bP$, a ray $\cong [0,+\infty)$ to the vertex $\widehat{\iota}(p)$. The ray corresponding to $p\in \bP$ is often referred to as the \textit{leg} of $p$. A \textit{stable $\bP$-marked genus-$0$ tropical curve} is a metrization of the topological space of a stable $\bP$-marked tree $\tree$, obtained by assigning a positive edge length to every edge of $\tree$, and assigning the legs infinite length. All tropical curves in this article are genus-$0$. 

The space $\cone^{\circ}(\tree)$ of $\bP$-marked tropical curves with underlying marked tree $\tree$ is a cone canonically isomorphic to $\R_{>0}^{\edges(\tree)}$. Here, the coordinates in terms of the natural basis are exactly the edge-lengths. We denote by $\cone(\tree)$ the space $\R_{\ge 0}^{\edges(\tree)}$ of possibly degenerate metrizations of $\tree$. If $\tree_2$ is obtained from $\tree_1$ by contracting a subset $E$ of edges of $\tree_1$, then a degenerate metrization of $\tree_1$ that assigns length $0$ to every $e\in E$ also yields a metrization of $\tree_2$. This realizes the cone $\cone(\tree_2)$ as a face of $\cone(\tree_1)$. Note that the direction of inclusion of cones is the reverse of the direction of inclusion of boundary strata outlined in Section \ref{sec:M0nbar}. The moduli space $\M_{0,\bP}^{\trop}$ is the space of all stable $\bP$-marked tropical curves; it is a cone complex obtained by gluing together the cones $\cone(\tree)$, where $\tree$ ranges over all stable $\bP$-marked trees, along the inclusions $\cone(\tree_2)\into \cone(\tree_1)$ mentioned above. As we explain in Section \ref{sec:nonalgebraictropical}, $\M_{0,\bP}^{\trop}$ can be identified with the cone over the boundary complex of $\Mbar_{0,\bP}$. 

Suppose that $\bP_1\subset\bP_2$ and that $\Sigma$ is a $\bP_2$-marked tropical curve. Set $\mu^{\trop}(\Sigma)$ to be the convex hull, within $\Sigma$, of the legs 
corresponding to $p\in \bP_1$; this is a $\bP_1$-marked tropical curve. The map $\mu^{\trop}:\M_{0,\bP_2}^{\trop}\to \M_{0,\bP_1}^{\trop}$ is piecewise linear, and is the tropicalization of the forgetful map $\mu$.


\section{Teichm\"uller space and its tropicalization}\label{sec:teich}

We introduce some objects from surface topology and propose a candidate for the role of tropical Teichm\"uller space. Given $\bP\subset S^2$, the Teichm\"uller space $\cT(S^2,\bP)$ is a non-algebraic complex manifold of dimension $\abs{\bP}-3$, parameterizing complex structures on $(S^2,\bP)$. More precisely, a point of $\cT(S^2,\bP)$ is an equivalence class of homeomorphisms $\psi:S^2\to \CC\PP^1$, where $\psi_1$ and $\psi_2$ are equivalent if there is a M\"obius transformation $\alpha$ such that $\psi_1$ and $\alpha\circ\psi_2$ are isotopic relative to $\bP$.

The pure mapping class group $\Mod(S^2,\bP)$ is the group of self-homeomorphisms of $S^2$ that fix every element of $\bP$, up to isotopy relative to $\bP$. $\Mod(S^2,\bP)$ acts freely and properly on $\cT(S^2,\bP)$ by precomposition \cite{FarbMargalitPrimer}. There is a forgetful map $\rho:\cT(S^2,\bP)\to \M_{0,\bP}$ sending $[\psi:S^2\to \CP]$ to $[\psi|_{\bP}]$; this map is the quotient by the action of $\Mod(S^2,\bP)$, and realizes $\cT(S^2,\bP)$ as the universal cover of $\M_{0,\bP}$. 

If $\bP_1\subset \bP_2 \subset S^2$, there is a holomorphic forgetful map $\widetilde{\mu}:\cT(S^2,\bP_2)\to\cT(S^2,\bP_1)$. The map $\widetilde{\mu}$ is a lift of the forgetful map $\mu:\M_{0,\bP_2}\to\M_{0,\bP_1}$.

\subsection{Multicurves and the curve complex}\label{sec:multicurves}

A simple closed curve on $S^2\setminus \bP$ is called \textit{essential} if it does not bound a disc, and is called \textit{non-peripheral} if it does not bound a punctured disc. A multicurve on $S^2\setminus \bP$ is a (possibly empty) set of homotopy classes of pairwise disjoint, essential and non-peripheral simple closed curves on $S^2\setminus \bP$.  

The curve complex $\CCom(S^2,\bP)$ of $(S^2,\bP)$ is the simplicial complex defined as follows: vertices of $\CCom(S^2,\bP)$ are in bijection with homotopy classes of simple closed curves on $S^2\setminus \bP$ that are essential and non peripheral. The vertices corresponding to $[\gamma_1],\ldots, [\gamma_{k+1}]$ span a $k$-simplex if and only if there exist pairwise disjoint representatives for $[\gamma_1],\ldots, [\gamma_{k+1}]$, i.e. if and only if $\{[\gamma_1],\ldots, [\gamma_{k+1}]\}$ is a multicurve on $S^2\setminus \bP$. The curve complex is locally infinite, and Gromov hyperbolic \cite{MasurMinskyCurveComplex}. The mapping class group $\Mod(S^2,\bP)$ acts on $\CCom(S^2,\bP)$, but this action can have infinite stabilizers.

\subsection{Augmented Teichm\"uller space}\label{sec:augteich} In order to tell a more complete story we introduce augmented Teichm\"uller space $\Tbar(S^2,\bP)$, a partial compactification of $\cT(S^2,\bP)$ parameterizing isotopy classes of tuples $(C,\iota, \psi)$, where $(C,\iota)$ us a $\bP$-marked stable curve and $\psi:S^2\to C$ is a surjective continuous map satisfying:
\begin{itemize}
    \item $\psi$ is a homeomorphism away from nodes of $C$, and the preimage of every node of $C$ is a simple closed curve of $S^2$, and
    \item $\psi|_{\bP}=\iota$.
\end{itemize}

The action of the mapping class group $\Mod(S^2,\bP)$ on $\cT(S^2,\bP)$ extends continuously to an action on $\Tbar(S^2,\bP)$; the quotient is $\Mbar_{0,\bP}$. The quotient map $\rho:\Tbar(S^2,\bP)\to \Mbar_{0,\bP}$ sends $[C,\iota, \psi]$ to the stable curve $[C,\iota]$. Note that given $[C,\iota, \psi]\in \Tbar(S^2,\bP)$, the set $\Gamma$ of homotopy classes of simple closed curves of $S^2$ contracted to nodes by $\psi$ is a multicurve on $S^2\setminus \bP$. The boundary $\Tbar(S^2,\bP)\setminus\cT(S^2,\bP)$ is stratified, with strata indexed by multicurves. The boundary complex of $\Tbar(S^2,\bP)$ (defined in Section \ref{sec:nonalgebraictropical}) is exactly the curve complex $\CCom(S^2,\bP)$.

\subsection{``Tropical Teichm\"uller space": the space of weighted multicurves}\label{sec:tropicalteichmuller} Given a multicurve $\Gamma$, we denote by $\cone(\Gamma)$ the cone $\R^{\Gamma}_{\ge 0}$. If there is a containment $\Gamma_2\subset \Gamma_1$ of multicurves on $S^2\setminus \bP$, then there is a natural inclusion, as a face, of $\cone(\Gamma_2)$ into $\cone(\Gamma_1)$. We set $\ConeCC(S^2,\bP)$ to be the union of the cones $\cone(\Gamma)$, where $\Gamma$ ranges over all multicurves on $S^2\setminus \bP$, glued along the aforementioned inclusion maps. This is the space of all ``multicurves with positive weights". Note that $\ConeCC(S^2,\bP)$ can be naturally identified with the cone over $\CCom(S^2,\bP)$, the boundary complex of $\Tbar(S^2,\bP)$. We propose here that $\ConeCC(S^2,\bP)$ is the correct candidate for the ``tropical Teichm\"uller space" $\cT(S^2,\bP)^{\trop}$. (There is currently no good framework for proving such a statement; however, in Section \ref{sec:nonalgebraictropical}, we attempt a justification for our proposal.)

Suppose that $\Gamma$ is a multicurve on $S^2\setminus\bP$. We pick pairwise disjoint representative curves $\gamma\in[\gamma]\in\Gamma$ and abuse notation by writing ``$S^2\setminus\Gamma$" to mean $S^2\setminus \bigcup_{[\gamma]\in\Gamma} \gamma$. The dual tree of $\Gamma$ is the stable $\bP_1$-marked tree $\tree(\Gamma)=(\underline{\tree}, \widehat{\iota})$ obtained as follows: vertices of $\underline{\tree}$ correspond to the connected components of $S^2\setminus \Gamma$. Edges of $\underline{\tree}$ are in canonical bijection with $\Gamma$: two vertices of $\underline{\tree}$ are joined by an edge $e([\gamma])$ if the closures of the corresponding connected components of $S^2\setminus \Gamma$ intersect exactly along $\gamma$. For $p\in\bP$, $\widehat{\iota}(p)$ is the vertex of $\underline{\tree}$ corresponding to the connected component of $S^2\setminus \Gamma$ containing $p$. Clearly, there is a canonical isomorphism of cones

\begin{align*}
\rho^{\trop}:\cone(\Gamma)=\R_{\ge0}^{\Gamma}&\to\cone(\tree(\Gamma))=\R_{\ge0}^{\edges(\tree(\Gamma))}\\
\sum_{[\gamma]\in\Gamma} a_{[\gamma]}[\gamma]&\mapsto \sum_{[\gamma]\in\Gamma} a_{[\gamma]}e([\gamma])
\end{align*}

These isomorphisms glue together to yield a continuous piecewise linear surjective map $\rho^{\trop}:\ConeCC(S^2,\bP) \to\M_{0,\bP}^{\trop}$. This map realizes $\M_{0,\bP}^{\trop}$ as the quotient of $\ConeCC(S^2,\bP)$ by $\Mod(S^2,\bP)$.

If $\bP_1\subset \bP_2\subset S^2$, a simple closed curve that is essential and non-peripheral on $S^2\setminus \bP_2$ may be non-essential or peripheral on $S^2\setminus \bP_1$. Also, two curves that are non-homotopic on $S^2\setminus \bP_2$ may be homotopic on $S^2\setminus \bP_1$. For $\Gamma'$ a multicurve on $S^2\setminus \bP_2$, we obtain a (possibly empty) multicurve on $S^2\setminus \bP_1$, denoted by $\widetilde{\mu}_{*}(\Gamma')$, as follows: throw away those $[\gamma']\in \Gamma'$ that are non-essential or peripheral on $S^2\setminus \bP_1$, and consider curves up to homotopy on $S^2\setminus \bP_1$ rather than on $S^2\setminus \bP_2$. There is a natural linear map of cones

\begin{align*}
\widetilde{\mu}^{\trop}:\cone(\Gamma')&\to\cone(\widetilde{\mu}_{*}(\Gamma'))\\
\sum_{[\gamma']\in\Gamma'} a_{[\gamma']}[\gamma']&\mapsto \sum_{[\gamma]\in\widetilde{\mu}_{*}(\Gamma')} \left( \sum_{\substack{[\gamma']\in\Gamma'\\\gamma' \text{homotopic to } \gamma\\\text{on }S^2\setminus\bP^1}}  a_{[\gamma']} \right)[\gamma]
\end{align*}
 
 The above maps $\cone(\Gamma')\to \cone(\widetilde{\mu}_{*}(\Gamma'))$ glue together to determine a continuous piecewise linear map  $\widetilde{\mu}^{\trop}:\ConeCC(S^2,\bP_2)\to\ConeCC(S^2,\bP_1)$. We suggest that this map is a good candidate for the tropicalization of the forgetful map $\widetilde{\mu}$. Note that there is no well-defined forgetful map from $\CCom(S^2,\bP_2)$ to $\CCom(S^2,\bP_1)$. 

\begin{rem} Note that $\widetilde{\mu}^{\trop}$ as defined above is a natural lift of $\mu^{\trop}:\M_{0,\bP_2}^{\trop}\to \M_{0,\bP_1}^{\trop}$. 
\end{rem}

\subsection{Tropicalization, within and without algebraic geometry}\label{sec:nonalgebraictropical}

Here, we attempt to provide more justification for the proposal (Section \ref{sec:tropicalteichmuller}) that
the space of weighted multicurves be considered a good candidate for  ``tropical Teichm\"uller space". In the context of algebraic geometry, the term ``tropicalization" has a few different, but related, meanings. We discuss some of them below (not in their most general senses).  

\begin{enumerate} 
    \item \textbf{Tropicalizations from embeddings in tori.} The most traditional and unambiguous meaning arises in the context of toric geometry. Suppose that $X$ is a closed subvariety of $(\C^{*})^N$. Then its tropicalization $X^{\trop}$ is a polyhedral complex in $\R^N$, see the textbook \cite{MaclaganSturmfels2015} for a definition.\label{it:torustrop}
    
    \item \textbf{Tropicalizations from compactifications.} \label{it:boundarytrop}Suppose $X$ is a smooth, non-compact variety admitting a smooth compactification $\overline{X}$ with the property that the boundary $\overline{X}\setminus X$ is a simple normal crossings hypersurface. Then the \textit{boundary complex of $\overline{X}$} is a simplicial complex that has a vertex for every irreducible component of the boundary, and $k$ vertices span a simplex if and only if the corresponding $k$-fold intersection of hypersurfaces is non-empty. The homotopy type of the boundary complex depends only on $X$, not on the choice of compactification $\overline{X}$ \cite{Thuillier2007, payne2013boundary}. The cone $X^{\trop}$ over the boundary complex of $\overline{X}$ is sometimes referred to as a tropicalization of $X$, see for example \cite{AbramovichCaporasoPayne2015} (including for a connection to the \textit{Berkovich analytification} of $X$). The cone complex encodes the asymptotics of $1$-parameter families of points in $X$ in the following way: A point $x(t)$ of $X$ defined over the field $\C((t))$ of formal Laurent series determines a point $x^{\trop}$ of $X^{\trop}$. The point $x^{\trop}$ lies in the relative interior of the cone of $X^{\trop}$ that corresponds to the boundary stratum of $\overline{X}$ whose relative interior contains the ``$t=0$ limit" of $x(t)$ in $\overline{X}$; the coordinates of $x^{\trop}$ inside that cone are the $t$-adic valuations of local equations for the boundary hypersurfaces. 

    \item \textbf{Tropicalizations from degenerations.}\label{it:dgeneratetrop} Suppose $(C(t), p_1(t),\ldots, p_n(t))$ is a $1$-parameter family of smooth algebraic curves defined over $\C((t))$. There is a tropical curve (metric graph) $C^{\trop}$ associated to $C(t)$; its underlying combinatorial graph is the dual graph of the stable curve obtained as the ``$t=0$ limit" $C(0)$ of $C(t)$, and edge-lengths are determined by the exponential rates of formation of nodes as $C(t)$ degenerates to $C(0)$. 
\end{enumerate}

We briefly summarize how the above notions apply to the case of moduli spaces of curves and their tropical counterparts. We include in our discussion topological and Riemann surfaces (and algebraic curves) of genus greater than $1$, their moduli spaces $\M_{g,n}$ and $\Mbar_{g,n}$, and Teichm\"uller and augmented Teichm\"uller spaces $\cT(g,n)$ and $\Tbar_{g,n}$, all defined analogously to the $g=0$ case discussed elsewhere in this article. We also need to include genus-$g$, $n$-marked tropical curves, which are obtained by assigning edge-lengths to the dual graphs of genus-$g$, $n$-marked stable curves. The moduli space of genus-$g$, $n$-marked tropical curves is denoted $\M_{g,n}^{\trop}$ \cite{AbramovichCaporasoPayne2015}. (Here, for uniformity, we use the notation $\M_{0,n}$ rather than $\M_{0,\bP}$.) The relationship between the algebraic and tropical moduli spaces is particularly satisfying in genus $0$ because all three notions of tropicalization apply, and interact seamlessly with one another. In the case of genus $g>0$, the first notion --- tropicalization from toric embedding --- does not apply, but the other two notions do apply with caveats, as we explain below. 

\begin{itemize}
\item There is a natural embedding of $\M_{0,n}$ into a torus $(\C^*)^{N}$ with respect to which the tropicalization --- in the sense of Item \ref{it:torustrop} above --- is $\M_{0,n}^{\trop}$. There is no such embedding of $\M_{g,n}$ into a torus for $g>0$. 
\item As mentioned earlier, $\M_{0,\bP}^{\trop}$ is the cone over the boundary complex of $\Mbar_{0,\bP}$. In positive genus, there are two complications: firstly, $\Mbar_{g,n}$ is an orbifold, not a complex manifold, and secondly, its boundary is normal crossings but not simple normal crossings. Both technicalities can be addressed, and there is a precise sense in which $\M_{g,n}^{\trop}$ is the cone over the boundary complex of $\Mbar_{g,n}$ \cite{CCUW, ChanGalatiusPayne2016}. 
\item A point $x(t)$ of $\M_{g,n}$ defined over $\C((t))$ determines a point $x^{\trop}$ of $\M_{g,n}^{\trop}$, as described in Item \ref{it:boundarytrop} above --- in this interpretation, $\M_{g,n}^{\trop}$ plays its role as boundary complex. The same point $x(t)$ also determines a smooth genus $g$ curve $X(t)$ defined over $\C((t))$ (or possibly over $\C((t^{\frac{1}{r}}))$), which admits a tropicalization $C^{\trop}$ in the sense of Item \ref{it:dgeneratetrop} above. Since $C^{\trop}$ is a tropical curve, it corresponds to a point of $\M_{g,n}^{\trop}$ --- in this interpretation, $\M_{g,n}^{\trop}$ plays its role as the space of tropical curves. One of the main results of \cite{AbramovichCaporasoPayne2015} is that the points of $\M_{g,n}^{\trop}$ determined by $x(t)$ with respect to these two different interpretations agree with each other. 
\end{itemize}

There is also a fourth connection between the algebro-geometric and the tropical moduli spaces of curves due to Chan-Galatius-Payne \cite{ChanGalatiusPayne2016}; this one departs from the realm of algebraic geometry:

\begin{itemize}
    \item There is a proper surjective continuous map $\lambda:\M_{g,n}\to \M_{g,n}^{\trop}$, natural up to homotopy. Conceptually, this is a non-algebraic version of the tropicalization map defined on ``dynamic" $\C((t))$-points as in Item \ref{it:boundarytrop}. Points of $\M_{0,n}$ --- defined over $\C$ and therefore ``static" --- that are ``close" to some boundary stratum of $\Mbar_{g,n}$ map to the cone of $\M_{g,n}^{\trop}$ corresponding to that stratum. The map $\lambda$ is described in Section 7.1 of \cite{ChanGalatiusPayne2016}; we give a brief summary here. Fix $\epsilon> 0$ sufficiently small so that on any genus-$g$, $n$-punctured hyperbolic surface, any two closed geodesics of length $\le\epsilon$ must be disjoint. Given $x\in \M_{g,n}$, we let $C_x$ be the genus-$g$, $n$-punctured Riemann surface representing $x$, together with its canonical hyperbolic metric. The collection $\Gamma$ of simple closed curves on $C_x$ with length strictly less than $\epsilon$ is a multicurve. We obtain a genus-$g$, $n$-marked tropical curve from the marked dual graph $G$ of $\Gamma$ by assigning length $-\log(\frac{\len(\gamma)}{\epsilon})$ to the edge of $G$ corresponding to $\gamma\in\Gamma$. The tropical curve thus obtained is the point $\lambda(x)\in\M_{g,n}^{\trop}$.
\end{itemize}


Moduli spaces of tropical curves were studied in other guises before the birth of tropical geometry, most notably in mathematical biology and geometric group theory. For example, $\M_{0,n}^{\trop}$ was studied in \cite{BilleraHolmesVogtmann} as the space of phylogenetic trees, and (a certain open subset of) $\M_{g,n}^{\trop}$ was used \cite{SmillieVogtmann} as a tool to compute the Euler characteristic of the outer automorphism group of the free group. There has increasingly been communication between tropical geometry and geometry outside algebraic geometry. It is in this context that we believe that the term ``tropicalization" might have even wider applications. None of the existing definitions of ``tropicalization" seem to apply to Teichm\"uller space. It would be very interesting to have precise definitions that might apply, but we do not attempt to come up with any here. Instead, we make an argument that the relationship between Teichm\"uller space $\cT_{g,n}$ and the cone $\ConeCC(g,n)$ over the curve complex (for a genus-$g$, $n$-marked surface) is analogous to the relationship between $\M_{g,n}$ and $\M_{g,n}^{\trop}$. 

\begin{itemize}
\item The pure mapping class group $\Mod(g,n)$ acts on $\cT_{g,n}$ and $\ConeCC(g,n)$; the quotient of $\cT_{g,n}$ by $\Mod(g,n)$ is $\M_{g,n}$ and the quotient of $\ConeCC(g,n)$ by $\Mod(g,n)$ is $\M_{g,n}^{\trop}$. 
\item As in Item \ref{it:boundarytrop} just above, $\M_{g,n}^{\trop}$ is the cone over the boundary complex of the normal-crossings compactification $\Mbar_{g,n}$ of $\M_{g,n}$. In contrast, $\Tbar_{g,n}$ is not a compactification of $\cT_{g,n}$, and the complex structure of $\cT_{g,n}$ does not extend to the boundary, and so the boundary of $\Tbar_{g,n}$ is in no way normal crossings. However, in a combinatorial sense, the boundary is as good as a simple normal crossings boundary --- the boundary is (complex) codimension-$1$ and stratified, and any non-empty $k$-fold intersection of codimension-$1$ strata has codimension exactly $k$. As a result, there is a well-defined notion of boundary complex that applies --- the boundary complex of $\Tbar_{g,n}$ is the simplicial complex that has a vertex for every closed codimension-$1$ stratum, and $k$ vertices span a simplex if and only if the corresponding codimension-$1$ strata have non-empty intersection. It is straightforward from the definition of $\Tbar_{g,n}$ (Section \ref{sec:augteich}) that the boundary complex of $\Tbar_{g,n}$ can be canonically identified with the curve complex $\CCom(g,n)$, and so
the cone over the boundary complex is exactly $\ConeCC(g,n)$.
\item The non-algebro-geometric continuous proper map $\lambda: \M_{g,n}\to \M_{g,n}^{\trop}$ of Chan-Galatius Payne (discussed above) lifts to a $\Mod(g,n)$-equivariant continuous (but non-proper) map $\tilde{\lambda}:\cT_{g,n}\to \ConeCC(g,n)$; this is transparently clear from the description of $\lambda$ given above.
\end{itemize}

\begin{rem}
     There is a compelling related story about \textit{Schottky space}, a non-algebraic complex manifold that can be realized as the quotient of Teichm\"uller space by a certain subgroup of the mapping class group. It is established in \cite{HerrlichGerritzen, poineau2022schottky, ulirsch2021non, ChanMeloViviani} that the tropicalization of Schottky space is the simplicial completion of \textit{Culler-Vogtmann Outer Space} \cite{CullerVogtmann}. For example, there is a partial compactification of Schottky space whose boundary complex is the completion of Outer Space. There seems to be at least one key difference between the cases of Schottky space and Teichm\"uller space --- the partial compactification of Schottky space extends its complex structure, whereas, as noted above, the complex structure of $\cT_{g,n}$ does not extend to $\Tbar_{g,n}$.
\end{rem}



\section{Thurston's topological characterization of PCF rational functions}\label{sec:ThurtonsTheorem}

In this section, we summarize the relevant parts of Thurston's results (written in \cite{DouadyHubbard1993}) giving a topological characterization of PCF rational functions. 

\subsection{PCF branched coverings and PCF rational functions}\label{sec:PCFcovers}


A map $\varphi:S^2 \to S^2$ is called a \textit{branched cover} if there is a finite set $\bB\subset S^2$ such that $\varphi: S^2\setminus \varphi^{-1}(\bB)\to S^2\setminus \bB$ is a covering map. Here, we will only deal with orientation-preserving branched covers. The restriction of $\varphi$ to a small punctured neighbourhood of any point $x\in S^2$ is a covering map of some degree; this is the \textit{local degree} of $\varphi$ at that point. A \textit{critical point} of $\varphi$ is a point at which the local degree is greater than $1$. Every rational function $f:\CC\PP^1\to \CC\PP^1$ is an orientation-preserving branched cover, and its critical points are exactly the points at which the derivative vanishes. The local degree of $f$ at $x$ is the unique $r>0$ such that the first $r-1$ derivatives of $f$ vanish at $f$ but the $r$th does not. 

The \textit{post-critical set} of a branched cover is the set $\{\varphi^n(x)\st n>0, x \mbox{ critical point of } \varphi \}$. A branched covering $\varphi: S^2\to S^2$ is called PCF (post-critically finite) if its post-critical set is finite. Two PCF branched covers $\varphi$ and $\varphi'$ with post-critical sets $\bP$ and $\bP'$ respectively are called \textit{combinatorially equivalent} (also referred to as \textit{Thurston equivalent}) if there are orientation-preserving homeomorphisms $\psi_1, \psi_2: S^2\to S^2$ such that 
\begin{itemize}
    \item $\psi_1$ and $\psi_2$ are isotopic relative to $\bP$, and both map $\bP$ onto $\bP'$, and
    \item $\varphi'=\psi_2\circ \varphi\circ \psi_1^{-1}$. 
\end{itemize}

Thurston's influential result answers the question: given a PCF branched cover $\varphi$, when is it combinatorially equivalent to a PCF rational function? In fact, if $\varphi$ has hyperbolic orbifold (as in footnote \ref{fn:hyperbolicorbifold}), then either (A) $\varphi$ is combinatorially equivalent to a rational map $f$, which is unique up to M\"obius conjugacy, or (B) $\varphi$ has a specific type of topological obstruction in the form of an $\varphi$-stable multicurve, defined in Section \ref{sec:obstructions}.

\subsection*{Fixing a branched cover} We now fix $\varphi$, a PCF branched cover with hyperbolic orbifold, of degree at least $2$, and whose post-critical set $\bP$ has cardinality at least $4$. Note that $\bP\subset \varphi^{-1}(\bP)$. We denote by $\widetilde{\mu}$ the forgetful map $\cT(S^2,\varphi^{-1}(\bP))\to\cT(S^2,\bP)$, denote by $\mu$ the forgetful map from $\M_{0,\varphi^{-1}(\bP})$ to $\M_{0,\bP}$, denote by $\rho$ the map $\cT(S^2,\bP)\to\M_{0,\bP}$, and denote by $\widetilde{\rho}$ the map $\cT(S^2,\varphi^{-1}(\bP))\to\M_{0,\varphi^{-1}(\bP)}$. 

\subsection{The pullback map on Teichm\"uller space}\label{sec:pullback}
A homeomorphism $\psi: S^2\to \CC\PP^1$ determines a complex structure on $S^2$. One can pull this complex structure back on charts along the branched cover $\varphi$ to obtain a new complex structure on $S^2$, and thus a global homeomorphism $\psi':S^2\to\CC\PP^1$; by construction, $\psi'\circ \varphi\circ \psi:\CC\PP^1\to\CC\PP^1$ is holomorphic. Although there are choices involved, the class of $\psi'$ in $\cT(S^2,\varphi^{-1}(\bP))$ is well-defined and only depends on the class of $\psi$ in $\cT(S^2,\bP)$. The map $\widetilde{\TP}_{\varphi}:\cT(S^2,\bP)\to \cT(S^2,\varphi^{-1}(\bP))$ that sends $[\psi]$ to $[\psi']$ is holomorphic. Composing with the forgetful map $\widetilde{\mu}$, we obtain a holomorphic selfmap $\TP_{\varphi}:=\widetilde{\mu}\circ\widetilde{\TP}_{\varphi}$ of $\cT(S^2,\bP)$, called \textit{Thurston's pullback map}.

\subsection{$\varphi$-stable multicurves and Thurston obstructions}\label{sec:obstructions}

If $\Gamma$ is a multicurve on $S^2\setminus \bP$, then its pullback $\widetilde{\varphi}^{*}(\Gamma):=\{[\varphi^{-1}(\gamma)]\st [\gamma]\in \Gamma\}$ is a multicurve on $S^2\setminus \varphi^{-1}(\bP)$. Some (or all) of the components components of $\widetilde{\varphi}^{*}(\Gamma)$ may be peripheral or non-essential on $S^2\setminus\bP$. We set $\varphi^{*}(\Gamma)$ to be the possibly empty multicurve $\widetilde{\mu}_{*}(\widetilde{\varphi}^{*}(\Gamma))$.


 There is a natural \textit{Thurston linear transformation} $\TLTtilde_{\varphi,\Gamma}$ from $\R^{\Gamma}$ to $\R^{\widetilde{\varphi}^{*}(\Gamma)}$, defined as follows in terms on its action on the natural basis:

\begin{align}
    \TLTtilde_{\varphi,\Gamma}([\gamma])= \sum_{\substack{\gamma' \text{simple closed curve}\\ \varphi(\gamma')=\gamma}} \frac{1}{\deg(\varphi:\gamma'\to\gamma)} [\gamma'] \label{eq:TropicalThurstonPullback1}
\end{align}

The composition $\TLT_{\varphi,\Gamma}:=\widetilde{\mu}^{\trop}\circ\TLTtilde_{\varphi,\Gamma}$ is also called the Thurston linear transformation; this is a map from $\R^{\Gamma}$ to $\R^{\varphi^{*}(\Gamma)}$ acting on the natural basis as follows:

\begin{align}
    \TLT_{\varphi,\Gamma}([\gamma])= \sum_{\substack{[\gamma''] \text{ essential and}\\ \text{non-peripheral rel. } \bP}} \sum_{\substack{\gamma' \text{simple closed}\\ \varphi(\gamma')=\gamma\\ \gamma' \text{ homotopic to }\gamma'' \text{ rel. } \bP}} \frac{1}{\deg(\varphi:\gamma'\to\gamma)} [\gamma''] \label{eq:TropicalThurstonPullback2}
\end{align}

Both linear transformations \eqref{eq:TropicalThurstonPullback1} and \eqref{eq:TropicalThurstonPullback2} are commonly referred to as \textit{Thurston linear transformations} associated to the multicurve $\Gamma$.

A multicurve $\Gamma$ is called \textit{$\varphi$-stable} if $\varphi^{*}(\Gamma)\subset\Gamma$. If $\Gamma$ is $\varphi$-stable then the Thurston linear transformation associated to $\Gamma$ can be considered to be a linear self-map of $\R^{\Gamma}$. Given that the Thurston linear transformation has non-negative coefficients with respect to the natural basis, its dominant eigenvalue $\Lambda_{\varphi}(\Gamma)$ is non-negative and has an eigenvector in $\R^{\Gamma}_{\ge 0}$. A $\varphi$-stable multicurve $\Gamma$ is called a \textit{Thurston obstruction} if $\Lambda_{\varphi}(\Gamma)\ge 1$. 

\subsection{Thurston's theorem} 

The pullback map $\TP_{\varphi}$ admits a fixed point if and only if $\varphi$ does not admit a Thurston obstruction. Also, $\TP_{\varphi}$ has a fixed point if and only if $\varphi$ is combinatorially equivalent to a rational function: if $[\psi]$ is a fixed point of $\TP_{\varphi}$, then there is a homeomorphism $\psi':S^2\to \CP$ that is isotopic to $\psi$ rel. $\bP$, and such that $\psi'\circ \varphi\circ \psi^{-1}$ is a PCF rational function. In fact, $\TP_{\varphi}$ is distance-nonincreasing and $\TP_{\varphi}^2$ is strictly distance-decreasing on $\cT(S^2,\bP)$, and so a fixed point, if it exists, is unique. We conclude that either (A) $\varphi$ is combinatorially equivalent to a rational map, which is unique up to M\"obius conjugacy, or (B) $\varphi$ admits a Thurston obstruction $\Gamma$.

\subsection{The augmented pullback map} \label{sec:augmentedpullback}Selinger \cite{selinger2012thurston} showed that Thurston's pullback map extends continuously to augmented Teichm\"uller space $\Tbar(S^2,\bP)$. We describe this extension in more detail in the proof of Lemma \ref{lem:combinatorialtypefrommulticurve}.

\subsection{``The tropical pullback map"}\label{sec:tropicalpullback} 

Observe that since the Thurston linear transformation \eqref{eq:TropicalThurstonPullback1} has non-negative coefficients with respect to the natural bases, it restricts to a map $\widetilde{\TP}_{\varphi}^{\trop}(\Gamma):\cone(\Gamma)\to \cone(\widetilde{\varphi}^{*}(\Gamma)).$
It is straightforward to check that the maps $\widetilde{\TP}_{\varphi}^{\trop}(\Gamma)$ glue together along faces to yield a piece-wise-linear map $\widetilde{\TP}_{\varphi}^{\trop}:\ConeCC(S^2,\bP)\to\ConeCC(S^2,\varphi^{-1}(\bP)).$ Similarly, the linear transformation \eqref{eq:TropicalThurstonPullback2} restricts to a map $\TP_{\trop}^{\trop}(\Gamma):\cone(\Gamma)\to \cone(\varphi^{*}(\Gamma)),$
and the maps $\TP_{\varphi}^{\trop}(\Gamma)$ glue together along faces to yield a piece-wise-linear map $\TP_{\varphi}^{\trop}:\ConeCC(S^2,\bP)\to\ConeCC(S^2,\varphi^{-1}(\bP)).$  We use the notation $\widetilde{\TP}_{\varphi}^{\trop}$ and $\TP_{\varphi}^{\trop}$ in order to propose that these piecewise linear maps ought to be considered to be tropicalizations of $\widetilde{\TP}_{\varphi}$ and $\TP_{\varphi}$ respectively. Note that $\TP_{\varphi}^{\trop}=\widetilde{\mu}^{\trop}\circ \widetilde{\TP}_{\varphi}^{\trop}.$ 

Under this interpretation, a $\varphi$-stable multicurve $\Gamma$ is exactly a multicurve such that $\cone(\Gamma)$ is invariant under the action of $\TP_{\varphi}^{\trop}$, in which case there is a ray in $\cone(\Gamma)$ that is fixed and scaled by the dominant eigenvalue $\Lambda_{\varphi}(\Gamma)$. A Thurston obstruction therefore is the same as a ray in $\ConeCC(S^2,\bP)$ that is fixed by $\TP_{\varphi}^{\trop}$, and scaled by a factor $\ge 1$.

\section{Hurwitz spaces and the moduli space correspondence}\label{sec:modulicorrespondence}

Here, we summarize the results of \cite{Koch2013} as they apply to this article. There is \textit{Hurwitz space} --- a non-compact smooth variety $\Hcomb(\varphi)$ parameterizing marked rational functions on $\CP$ that have the same branching profile as $\varphi$. More precisely, $\Hcomb(\varphi)$ parameterizes tuples $[f,\iota_1,\iota_2]$, where 
\begin{itemize}
    \item $\iota_1:\bP\into\CP$ and $\iota_2:\varphi^{-1}(\bP)\into\CP$ are injections,
    \item $f:(\CP,\iota_2(\varphi^{-1}(\bP)))\to(\CP,\iota_1(\bP))$ is a rational function that has the same degree as $\varphi$, and that maps $\iota_2(\varphi^{-1}(\bP))$ to $\iota_1(\bP)$ exactly as $\varphi$ maps $\varphi^{-1}(\bP)$ to $\bP$ (that is, $f\circ \iota_2=\iota_1\circ\varphi$), including with the same local degrees. (Note that this forces $f$ to have no critical values away from $\iota_1(\bP)$.)
\end{itemize}

Two tuples $[f,\iota_1,\iota_2]$ and $[f',\iota_1',\iota_2']$ are equal in $\Hcomb(\varphi)$ if they differ by changes of coordinates on $\CP$, i.e. if there are automorphisms $\alpha_1$ and $\alpha_2$ of $\CP$ such that $\iota_1'=\alpha_1\circ\iota_1$, $\iota_2'=\alpha_2\circ\iota_2$, and $f'=\alpha_1\circ f\circ(\alpha_2)^{-1}$. In general, $\Hcomb(\varphi)$ is disconnected, with connected components corresponding to a type of non-dynamical topological equivalence known as \textit{Hurwitz equivalence}, defined below.

\begin{Def}\label{def:hurwitzeq}
   Given $\iota_1',\iota_1'':\bP\into S^2$ and $\iota_2',\iota_2'':\varphi^{-1}(\bP)\into S^2$ injections,  and $\varphi':(S^2,\iota_2'(\varphi^{-1}(\bP)))\to(S^2,\iota_1'(\bP))$ and $\varphi'':(S^2,\iota_2''(\varphi^{-1}(\bP)))\to(S^2,\iota_1''(\bP))$ branched covers, we say that $\varphi'$ and $\varphi''$ are \textit{Hurwitz equivalent} if there exist homeomorphisms $\psi_1,\psi_2:S^2\to \CP$ such that $(\psi_1)\circ\iota_1'=\iota_1''$ on $\bP$, $\psi_2\circ\iota_2'=\iota_2''$ on $\varphi^{-1}(\bP)$, and $\varphi''=\psi_1\circ \varphi'\circ(\psi_2)^{-1}$. Note that $\psi_1$ and $\psi_2$ are \textbf{not} required to be isotopic rel. $\bP$.
\end{Def}

$\Hcomb(\varphi)$ has a unique connected component $\cH(\varphi)$ consisting of marked rational functions $f$ that are Hurwitz equivalent to $\varphi$. If $\varphi'$ is Hurwitz equivalent to $\varphi$ then $\cH(\varphi')=\cH(\varphi)$.

There are maps $\pi_1,\pi_2:\Hcomb(\varphi)\to\M_{0,\bP}$ send $[f,\iota_1,\iota_2]$ to $[\iota_1]$ and $[\iota_2|_{\bP}]$ respectively. The map $\pi_1$ is a covering map. Note that we can factor $\pi_2$ as $\pi_2=\mu\circ\widetilde{\pi}_2$, where $\widetilde{\pi}_2:\Hcomb(\varphi)\to\M_{0,\varphi^{-1}(\bP)}$ sends $[f,\iota_1,\iota_2]$ to $[\iota_2]$. The pair of maps $\pi_1,\pi_2:\cH(\varphi)\to\M_{0,\bP}$ as well as the multivalued self-map $$\MSC_{\varphi}:=(\pi_2)|_{\cH(\varphi)}\circ(\pi_1)|_{\cH(\varphi)}^{-1}$$ of $\M_{0,\bP}$ are often referred to as the \textit{moduli space correspondence}\footnote{The Hurwitz space $\cH(\varphi)$ defined here is a finite cover of the Hurwitz space introduced in \cite{Koch2013}, and so our moduli space correspondence is off from the standard moduli space correspondence by a global multiplicity factor. This minor technical modification doesn't affect any of the important properties but is better-behaved in the contexts of compactifications and tropicalizations of Hurwitz spaces.}. If $\varphi'$ is Hurwitz equivalent to $\varphi$ then $\MSC_{\varphi'}=\MSC_{\varphi}$. The holomorphic Thurston pullback map $\TP_{\varphi}$ descends, via the universal covering map $\rho$, to the algebraic moduli space correspondence on $\M_{0,\bP}$ in the following sense: there is a holomorphic covering map $\nu_{\varphi}:\cT(S^2,\bP)\to \cH(\varphi)$ such that $\rho=\pi_1\circ\nu_{\varphi}$, and $\rho\circ \TP_{\varphi}=\pi_2\circ\nu_{\varphi}$. The following diagram encapsulates the relationship between the moduli space correspondence and Thurston's pullback map. 

\begin{equation}\label{diag:noncompact}
  \begin{tikzpicture}
    \matrix(m)[matrix of math nodes,row sep=3em,column sep=8em,minimum
    width=2em] {
      \cT(S^2,\bP) &&\cT(S^2,\bP)\\
      &\cH(\varphi)&\\
      \M_{0,\mathbf{P}}&&\M_{0,\mathbf{P}}\\
    };
    \path[-stealth] (m-1-1) edge node [above] {$\TP_{\varphi}$} (m-1-3);
    \path[-stealth] (m-1-1) edge node [above right] {$\nu_{\varphi}$} (m-2-2);
    \path[-stealth] (m-1-1) edge node [left] {$\rho$} (m-3-1);
    \path[-stealth] (m-1-3) edge node [left] {$\rho$} (m-3-3);
    \path[-stealth] (m-2-2) edge node [above left]
    {$\pi_1$} (m-3-1);
    \path[-stealth] (m-2-2) edge node [above right] {$\pi_2$} (m-3-3);
  \end{tikzpicture}
\end{equation}
  A fixed point of the moduli space correspondence is a point $x=[f,\iota_1,\iota_2]\in \cH(\varphi)$ such that $[\iota_1]=\pi_1(x)=\pi_2(x)=[\iota_2|_{\bP}]$. Note this means that $[\iota_1]$ is only a ``weakly" fixed point of the multivalued map $\MSC_{\varphi}$ in the sense that it is contained in its own image. In general, $\MSC_{\varphi}$ is unlikely to have a ``strongly" fixed point, i.e. a point whose image contains itself and no other point. Fixed points of the moduli space correspondence correspond to PCF rational functions that are Hurwitz equivalent to $\varphi$. 

\section{Compactifications and tropicalizations of Hurwitz spaces}\label{sec:HurwitzCompactificationandTropicalization}

\subsection{Admissible covers}\label{sec:admissiblecovers}

An \textit{admissible cover}\footnote{This usage of the term \textit{admissible cover} to mean an algebraic map of nodal curves is standard in algebraic geometry. However it unfortunately conflicts with the usage of the same term in eg. \cite{KochPilgrimSelinger2016} to mean a topological branched cover of marked surfaces.} is a finite map $f:C_2\to C_1$ of pre-stable curves satisfying
\begin{enumerate}
    \item smooth points of $C_2$ map to smooth points of $C_1$, and nodes of $C_2$ maps to nodes of $C_1$, 
    \item (\textit{balancing at nodes}) the two branches at a node $x$ of $C_2$ map to distinct branches at $f(x)$, with equal local degrees at $x$.
\end{enumerate}

The above conditions ensure that $f$ is can be obtained as a limit of a family of maps between smooth algebraic curves.  An \textit{admissible cover with the same profile as $\varphi$} is a tuple $(f, C_1, \iota_1, C_2,\iota_2)$, where
\begin{itemize}
    \item $(C_1, \iota_1)$ is a $\bP$-marked stable curve,
    \item $(C_2, \iota_2)$ is a $\varphi^{-1}(\bP)$-marked stable curve,
    \item $f:C_2\to C_1$ is an admissible cover of the same degree as $\varphi$, such that
    \begin{itemize}
        \item $f$ maps $\iota_2(\varphi^{-1}(\bP))$ to $\iota_1(\bP)$ exactly as $\varphi$ maps $\varphi^{-1}(\bP)$ to $\bP$ (that is, $f\circ \iota_2=\iota_1\circ\varphi$), including with the same local degrees,
        \item $f$ has no critical points away from $\iota_2(\varphi^{-1}(\bP))$ and the set of nodes of $C_2$.
    \end{itemize}
\end{itemize}

Due to works of Harris-Mumford \cite{HarrisMumford1982}, there is a possibly reducible projective variety $\Hbarcomb(\varphi)$, with mild singularities, that parameterizes admissible covers with the same profile as $\varphi$. $\Hbarcomb(\varphi)$ contains $\cH(\varphi)$ as an open subset. Set $\Hbar(\varphi)$ to be the closure of $\cH(\varphi)$ in $\Hbarcomb(\varphi)$ -- this is an irreducible component of $\Hbarcomb(\varphi)$. If $[(f, C_1, \iota_1, C_2,\iota_2)]\in \Hbar(\varphi)$ we say that $f$ is an admissible cover of \textit{topological type $\varphi$}\footnote{There is in principle a way to characterize $f$ of topological type $\varphi$ in terms of the Hurwitz equivalence classes of the restrictions of $f$ to the irreducible $\CP$ components of $C_2$. However, it is subtle and, as far as we are aware, has not been carefully worked out in the literature.}. We set $\Htildecomb(\varphi)$ to be the normalization of $\Hbarcomb(\varphi)$ -- by \cite{HarrisMumford1982} $\Htildecomb(\varphi)$ is smooth, and by Abramovich-Corti-Vistoli \cite{AbramovichCortiVistoli2003} it admits an interpretation as a moduli space of covers of orbifold curves. $\Htildecomb(\varphi)$ contains as a connected component the normalization $\Htilde(\varphi)$ of $\Hbar(\varphi)$ --- this is a smooth compactification of $\cH(\varphi)$.

\subsection{The moduli space correspondence on $\Mbar_{0,\bP}$}\label{sec:compactifiedcorrespondence}

The maps $\pi_1$ and $\pi_2$ from $\Hcomb(\varphi)$ to $\M_{0,\bP}$ extend to regular maps from both $\Hbarcomb(\varphi)$ and its normalization $\Htildecomb(\varphi)$ to $\Mbar_{0,\bP}$. The extension of $\pi_1$ to $\Htildecomb(\varphi)$ is a ramified cover. As a result, the multivalued moduli space correspondence $\MSC_{\varphi}$ extends to $\Mbar_{0,\bP}$ ``without indeterminacy" \cite{Ramadas2015}. Also, just as $\pi_2$ factors as 
$\mu\circ\widetilde{\pi_2}$, so does its extension to $\Htildecomb(\varphi)$. By \cite{HinichVaintrobAugmented}, the holomorphic map $\nu_{\varphi}:\cT(S^2\bP)\to \cH(\varphi)$ extends continuously to a surjective map  $\nu_{\varphi}:\Tbar(S^2,\bP)\to\Hbar(\varphi)$, which in turn lifts to a continuous surjective map $\widetilde{\nu_{\varphi}}:\Tbar(S^2,\bP)\to\Htilde(\varphi)$. We describe the extension $\nu_{\varphi}$ in more detail in the proof of Lemma \ref{lem:combinatorialtypefrommulticurve}. The upshot is that Diagram \eqref{diag:noncompact} from Section \ref{sec:modulicorrespondence} extends to partial/full compactifications of the spaces involved as depicted below, and the compactified moduli space correspondence descends from the augmented pullback map.

\begin{equation}\label{diag:compact}
  \begin{tikzpicture}
    \matrix(m)[matrix of math nodes,row sep=3em,column sep=8em,minimum
    width=2em] {
      \Tbar(S^2,\bP) &&\Tbar(S^2,\bP)\\
      &\Hbar(\varphi)&\\
      \Mbar_{0,\mathbf{P}}&&\Mbar_{0,\mathbf{P}}\\
    };
    \path[-stealth] (m-1-1) edge node [above] {$\TP_{\varphi}$} (m-1-3);
    \path[-stealth] (m-1-1) edge node [above right] {$\nu_{\varphi}$} (m-2-2);
    \path[-stealth] (m-1-1) edge node [left] {$\rho$} (m-3-1);
    \path[-stealth] (m-1-3) edge node [left] {$\rho$} (m-3-3);
    \path[-stealth] (m-2-2) edge node [above left]
    {$\pi_1$} (m-3-1);
    \path[-stealth] (m-2-2) edge node [above right] {$\pi_2$} (m-3-3);
  \end{tikzpicture}
\end{equation}

\subsection{Combinatorial types of admissible covers}\label{sec:hurwitzboundarystratification}

Suppose $(f, C_1, \iota_1, C_2,\iota_2)$ is an admissible cover with the same profile as $\varphi$. The \textit{combinatorial type} $\Lambda=(\tree_1,\tree_2,F,\edgedeg, \varphi|_{\varphi^{-1}(\bP)}, \legdeg)$ of $f$ records 
\begin{itemize}
\item the marked dual trees $\tree_1=(\underline{\tree_1},\widehat{\iota_1})$ and $\tree_2=(\underline{\tree_2},\widehat{\iota_2})$ of $C_1$ and $C_2$ respectively,
\item the graph homomorphism $F:\underline{\tree_2}\to\underline{\tree_1}$ of dual trees that encodes on vertices how irreducible components of $C_2$ map to irreducible components of $C_1$, and on edges how nodes of $C_2$ map to nodes of $C_1$,
\item the \textit{edge-degree} map $\edgedeg:\edges(\tree_2)\to \Z_{>0}$: for $e$ an edge of $\tree_2$ corresponding to a node $\eta$ of $C_2$, $\edgedeg(e)$ is the local degree of $f$ at $\eta$ (this is well-defined by the balancing condition), 
\item the restriction $\varphi|_{\varphi^{-1}(\bP)}$ recording how legs of $\tree_2$ map to legs of $\tree_1$, and 
\item the \textit{leg-degree} map $\legdeg:\varphi^{-1}(\bP)\to \Z_{>0}$ sending $p$ to the local degree of 
$\varphi$ at $p$ (equivalently, that of $f$ at $\iota_2(p)$). 
\end{itemize} 

Note that the final two entries of the combinatorial type of $f$ depend only on $\varphi$. If $\Lambda$ is the combinatorial type of some admissible cover with the same profile as $\varphi$, we will say that $\Lambda$ is a \textit{realizeable combinatorial type of profile $\varphi$}. If $\Lambda$ is the combinatorial type of some admissible cover with the same topological type as $\varphi$, we will say that $\Lambda$ is a \textit{realizeable combinatorial type of topological type $\varphi$}.

The locus $\cH(\Lambda)\subset \Hbarcomb(\varphi)$ consisting of admissible covers with fixed combinatorial type $\Lambda$ is a locally-closed subset that is isomorphic to a product of Hurwitz spaces.  Its closure $\Hbar(\Lambda)$ is a not-necessarily-irreducible subvariety called a boundary stratum. The boundary stratum $\Hbar(\Lambda)$ has a component contained in $\Hbar(\varphi)$ if and only if $\Lambda$ has topological type $\varphi$. 

Suppose that $\Lambda=(\tree_1,\tree_2,F,\edgedeg, \varphi|_{\varphi^{-1}(\bP)}, \legdeg)$ is a realizeable combinatorial type of profile $\varphi$, and that $E$ is a (possibly empty) subset of edges of $\tree_1$. Set $\tree_1'$ to be the tree obtained from $\tree_1$ by contracting all the edges not in $E$, and set $\tree_2'$ to be the tree obtained from $\tree_1$ by contracting all the edges $e$ such that $F(e)\not\in E$. There is an induced graph homomorphism $F':\underline{\tree_2'}\to \underline{\tree_1'}$. Set $\Lambda'=(\tree_1',\tree_2',F',\edgedeg|_{E}, \varphi|_{\varphi^{-1}(\bP)}, \legdeg)$. Then $\Lambda'$ is also a realizeable combinatorial type of profile $\varphi$; in fact, any admissible cover $f$ with combinatorial type $\Lambda$ can be obtained as a limit of a family of admissible covers with combinatorial type $\Lambda'$ \cite{HarrisMumford1982}. In this setting we say that $\Lambda'$ is obtained from $\Lambda$ by edge-contraction. The containment $\Hbar(\Lambda)\subset\Hbar(\Lambda')$ holds if and only if $\Lambda'$ is obtained from $\Lambda$ by edge-contraction. If $\Lambda$ has topological type $\varphi$ and $\Lambda'$ is obtained from $\Lambda$ by edge-contraction then $\Lambda'$ also has topological type $\varphi$.

\begin{rem}\label{rem:topologicaltype}
It is possible but computationally difficult to enumerate the realizable combinatorial types of profile $\varphi$.  Doing so involves solving cases of the \textit{Hurwitz realizability problem} of deciding whether there exist branched covers with a specified ramification profile, or, equivalently, whether one can factor the identity element in a symmetric group as a product of elements with specified cycle types. However, it is so far not well-understood how to determine whether or not a realizable combinatorial type has topological type $\varphi$. 
\end{rem}

\subsection{The tropical Hurwitz space}\label{sec:tropicalhurwitz}

In this section, we summarize some of \cite{CavalieriMarkwigRanganathan2016} as it applies to the moduli space correspondence. Suppose that $\Lambda=(\tree_1,\tree_2,F,\edgedeg, \varphi|_{\varphi^{-1}(\bP)}, \legdeg)$ is a realizeable combinatorial type of profile $\varphi$. A \textit{tropical admissible cover} of combinatorial type $\Lambda$ is a tuple $(\Sigma_1,\Sigma_2,\mathbf{F})$, where:
\begin{itemize}
    \item $\Sigma_1$ and $\Sigma_2$ are tropical curves ($\bP$- and $\varphi^{-1}(\bP)$- marked respectively) with underlying trees $\tree_1$ and $\tree_2$ respectively,
    \item $\mathbf{F}:\Sigma_2\to \Sigma_1$ is a piecewise affine map realizing the combinatorial graph homomorphism $F$, such that the leg of $\Sigma_2$ corresponding to $p\in\varphi^{-1}(\bP)$ maps to the leg of $\Sigma_1$ corresponding to $\varphi(p)$, and such that slopes encode local degrees as follows:
    \begin{itemize}
    \item for $e$ an edge of $\Sigma_2$, $\slope(\mathbf{F}|_{e})=\edgedeg(e)$, and
    \item for $l$ a leg of $\Sigma_2$ corresponding to $p\in \varphi^{-1}(\bP)$, $\slope(\mathbf{F}|_{l})=\text{(local deg. of $\varphi$ at $p$)}$.
    \end{itemize}
\end{itemize}

Given $\Lambda$ together with $\Sigma_1$ a tropical curve with underlying tree $\tree_1$, there is a unique tropical admissible cover $(\Sigma_1,\Sigma_2,\mathbf{F})$ with combinatorial type $\Lambda$: the edge-lengths of $\Sigma_2$ and the piecewise affine structure of $\mathbf{F}$ are determined by $F$, $\edgedeg$ and $\legdeg$. This tells us that the space $\cone^{\circ}(\Lambda)$ of tropical  admissible covers of combinatorial type $\Lambda$ is a cone isomorphic to $\R_{>0}^{\abs{\edges(\tree_1)}}$. However, unlike the case of a cone $\cone^{\circ}(\tree)=\R_{>0}^{\edges(\tree)}$ of $\M_{0,\bP}^{\trop}$, where the cone coordinates are exactly the edge lengths of the tropical curve, in the case of $\cone^{\circ}(\Lambda)$ the relationship between the cone coordinates and the edge-lengths of $\Sigma_1$ is more complicated, as we describe. In order to avoid confusion, we will use the notation $\check{e}$ for the cone coordinate on $\cone^{\circ}(\Lambda)$ corresponding to $e\in\edges(\tree_1)$. Now, for $e\in \edges(\tree_1)$, set $\lcmdeg(e):=\lcm \{\edgedeg(e')\}_{F(e')=e}$. Then, from $\sum_{e\in\edges(\tree_1)} a_e \check{e}\in\cone^{\circ}(\Lambda)$ we obtain a tropical curve $\Sigma_1$ by assigning, to edge $e$ of $\tree_1$, the length $\lcmdeg(e)\cdot a_e$. This edge-length assignment determines a unique tropical admissible cover of combinatorial type $\Lambda$. It follows that an edge $e'$ of $\tree_2$ mapping via $F$ to $e$ is assigned length $\frac{\lcmdeg(e)}{\edgedeg(e')}\cdot a_e$ (\cite{CavalieriMarkwigRanganathan2016} Section 5.4). The closed cone $\cone(\Lambda)=\R_{\ge0}^{\{\check{e}\}}$ can be canonically identified with the space of of tropical admissible covers whose combinatorial types can be obtained from $\Lambda$ by edge-contraction. Set $\Htropcomb(\varphi)$ to be the \textit{tropical Hurwitz space} parameterizing tropical admissible covers of profile $\varphi$; this is a complex constructed by gluing the cones $\{\cone(\Lambda)\}_{\Lambda \mbox{ realizeable}}$ together along the inclusions of faces $\cone(\Lambda')\into\cone(\Lambda)$, where $\Lambda'$ is obtained from $\Lambda$ by edge-contraction. 

There are natural piecewise linear maps $\pi_1^{\trop}$ and $\pi_2^{\trop}$ from $\Htropcomb(\varphi)$ to $\M_{0,\bP^{\trop}}$, sending $(\Sigma_1,\Sigma_2,\mathbf{F})$ to $[\Sigma_1]$ and $[\mu^{\trop}(\Sigma_2)]$ respectively; clearly $\pi_2^{\trop}$ factors as $\mu^{\trop}\circ\widetilde{\pi_2}^{trop}$, where $\widetilde{\pi_2}^{\trop}:\Htropcomb(\varphi)\to \M_{0,\varphi^{-1}(\bP)}$ sends $(\Sigma_1,\Sigma_2,\mathbf{F})$ to $[\Sigma_2]$. The map $\pi_1^{\trop}$ is finite-to-one. More specifically, given any cone $\cone(\tree_1)$ of $\M_{0,\bP}^{\trop}$, $(\pi_1^{trop})^{-1}(\cone(\tree_1))$ consists of finitely many cones of $\Htropcomb(\varphi)$. Each such cone is of the form $\cone(\Lambda)$, where $\Lambda=(\tree_1,\ldots)$. In terms of the generators $(\check{e})_{e\in \edges(\tree_1)}$ for $\cone(\Lambda)$ and $\edges(\tree_1)$ for $\cone(\tree_1)$, the map $\pi_1^{\trop}$ is diagonal with scaling factors $(\lcmdeg(e))$. This reflects the local geometry of the map $\pi_1:\Htildecomb(\varphi)\to\Mbar_{0,\bP}$ as follows: there are local coordinates $(x_e)$ on $\Htildecomb(\varphi)$ along the boundary stratum $\Hbar(\Lambda)$ and local coordinates $(y_e)$ on $\Mbar_{0,\bP}$ along $\stratum_{\tree_1}$ such that $\pi_1^{*}(y_e)=x_e^{\lcmdeg(e)}$.

\begin{rem}[Relationship with the boundary complex of admissible covers]\label{rem:truetropicalhurwitz}
    Unlike $\M_{0,\bP}^{\trop}$, which is the cone over the boundary complex of $\Mbar_{0,\bP}$, the tropical Hurwitz space $\Htropcomb(\varphi)$ is not the cone over the boundary complex of $\Htildecomb(\varphi)$. There is a possibly disconnected cone complex $\cH^{\trop}_{\mathrm{prof, true}}$, the ``true tropical Hurwitz space", that has a $k$-dimensional cone for every irreducible component of some codimension-$k$ boundary stratum of $\Htildecomb(\varphi)$. By \cite{CavalieriMarkwigRanganathan2016}, there is a finite-to-one map $\chi:\cH^{\trop}_{\mathrm{prof,true}}\to \Htropcomb(\varphi)$ whose restriction to every open cone is an isomorphism onto its image, but which might send distinct cones of $\cH^{\trop}_{\mathrm{prof,true}}$ to the same cone of $\Htropcomb(\varphi)$. 
\end{rem}

\section{The tropical moduli space correspondence and Thurston's pullback map}\label{sec:tropicalcorrespondence}

 There is a closed sub-cone-complex (of full dimension) $\Htrop(\varphi)\subset\Htropcomb(\varphi)$ consisting of tropical admissible covers whose combinatorial types also have topological type $\varphi$. However, as explained in Remark \ref{rem:topologicaltype} it is not well-understood how to in general enumerate the cones of, and to explicitly describe  $\Htrop(\varphi)$. We will refer to pair of maps $\pi_1^{\trop}, \pi_2^{\trop}:\Htrop(\varphi)\to \M_{0,\bP}^{\trop}$ as well as to the multivalued map $$\MSC_{\varphi}^{\trop}:=(\pi_2^{\trop})|_{\Htrop(\varphi)}\circ(\pi_1^{\trop})|_{\Htrop(\varphi)}^{-1}:\M_{0,\bP}^{\trop}\to \M_{0,\bP}^{\trop}$$ as the \textit{naive tropical moduli space correspondence}. 

\begin{rem}\label{rem:truetropicalcorrespondence}
    The ``true" tropical Hurwitz space $\cH^{\trop}_{\mathrm{prof,true}}$ (Remark \ref{rem:truetropicalhurwitz}) has a connected component $\cH^{\trop}_{\mathrm{true}}:=\chi^{-1}(\Htildecomb(\varphi))$ that is the cone over the boundary complex of $\Htilde(\varphi)$. The pair of maps $\pi_1^{\trop}\circ\chi, \pi_2^{\trop}\circ\chi:\Htrop_{\mathrm{true}}(\varphi)\to \M_{0,\bP}^{\trop}$ ans well as the multivalued map $\pi_2^{\trop}\circ\chi\circ((\pi_1^{\trop}\circ\chi)|_{\Htrop_{\mathrm{true}}(\varphi)})^{-1}$ ought to be thought of as the ``true" tropical moduli space correspondence. Given that the difference between $\cH^{\trop}_{\mathrm{true}}$ and $\Htilde(\varphi)$ lies purely in the multiplicity of occurrence of cones, the difference between the true and the naive tropical moduli space correspondence also lies purely in multiplicities. If the global dynamics of the tropical moduli space correspondence are to be studied in any substantive way, it is absolutely necessary to get the multiplicities correct. This study therefore ought to be carried out in the context of the true correspondence -- we believe that this is the object that best captures the algebraic dynamics ``near infinity" of the moduli space correspondence on $\M_{0,\bP}$. 
    
    In the case that $\varphi$ has topological degree $2$, the situation is particularly good. The Hurwitz space $\Hcomb(\varphi)$ is connected, so in fact $\cH(\varphi)=\Hcomb(\varphi)$. The compactifications $\Hbarcomb(\varphi)$, $\Hbar(\varphi)$, $\Htildecomb(\varphi)$, and $\Htilde(\varphi)$ all coincide, and have irreducible boundary strata. This implies that the tropical spaces $\Htrop(\varphi)$, $\Htropcomb(\varphi)$, $\Htrop_{\mathrm{true}}(\varphi)$, and $\cH^{\trop}_{\mathrm{prof, true}}$ coincide as well, and can be concretely described. In the context that $\varphi$ is degree-$2$, one is comparatively well-placed to study the dynamics if the true moduli space correspondence. 
\end{rem}

\subsection{Tropical admissible covers from weighted multicurves}

In this section, we establish that the relationship between the tropical moduli space correspondence and the tropical Thurston pullback map exactly mirrors the relationship between the algebraic moduli space correspondence and the holomorphic Thurston pulllback map. First, we show how to obtain a combinatorial type of admissible cover given a multicurve on $S^2\setminus \bP$. 
\begin{Def}\label{def:combinatorialtypefrommulticurve}
    Suppose that $\Gamma=\{\gamma_1,\ldots,\gamma_k\}$ is a multicurve on $S^2\setminus\bP$. We obtain a combinatorial type $$\Lambda_{\varphi}(\Gamma)=(\tree_1,\tree_2,F,\edgedeg, \varphi|_{\varphi^{-1}(\bP)}, \legdeg)$$ of admissible cover from $\Gamma$ and $\varphi$ as follows. Set $\tree_1=\tree(\Gamma)$ to be its marked dual tree as defined in Section \ref{sec:tropicalteichmuller}. Set $\Gamma_2:=\widetilde{\varphi}^{*}(\Gamma)$ to be its pullback (a multicurve on $S^2\setminus\varphi^{-1}(\bP)$), and set $\tree_2$ to be the $\varphi^{-1}(\bP)$-marked dual tree of $\Gamma_2$. Set $F:\underline{\tree_2}\to\underline{\tree_1}$ to be the graph homomorphism determined by $\varphi$ in the natural way, specifically:
    \begin{itemize}
        \item For $v$ a vertex of $\tree_2$ corresponding to connected component $Y$ of $S^2\setminus(\gamma)_{\gamma\in\Gamma_2}$, set $F(v)$ to be the vertex of $\tree_1$ corresponding to the connected component $\varphi(Y)$ of $S^2\setminus(\gamma)_{\gamma\in\Gamma}$. 
        \item For $e$ an edge of $\tree_2$ corresponding to simple closed curve $\gamma\in\Gamma_2$, set $F(e)$ to be the edge of $\tree_1$ corresponding to simple closed curve $\varphi(\gamma)\in \Gamma$. 
    \end{itemize}
Set $\edgedeg:\edges(\tree_2)\to \Z_{>0}$ be the map that sends the edge $e$ of $\tree_2$ corresponding to $\gamma\in\Gamma_2$ to $\deg(\varphi|_{\gamma})$. As explained in Section \ref{sec:hurwitzboundarystratification}, $\varphi|_{\varphi^{-1}(\bP)}$ and $\legdeg$ are determined by just $\varphi$.
\end{Def}


\begin{lem}[\cite{HinichVaintrobAugmented, selinger2012thurston}]\label{lem:combinatorialtypefrommulticurve}
    For any multicurve $\Gamma$, the combinatorial type $\Lambda_{\varphi}(\Gamma)$ in Definition \ref{def:combinatorialtypefrommulticurve} has topological type $\varphi$.
\end{lem}
\begin{proof}
    By \cite{Koch2013} the map $\nu_{\varphi}:\cT(S^2,\bP)\to \cH(\varphi)$ can be described in terms of the pullback map $\widetilde{\TP}_{\varphi}$ as follows. Given $[\psi]\in \cT(S^2,\bP)$, under the pair of complex structures $[\psi]$ and $\widetilde{\TP}_{\varphi}([\psi])$ on $(S^2,\bP)$ and $(S^2,\varphi^{-1}(\bP))$ respectively, the branched cover $\varphi$ is identified with a holomorphic map $f$; $\nu_{\varphi}$ sends $[\psi]$ to $[f]$.
    
    The extension of $\widetilde{\TP}_{\varphi}$ to $\Tbar(S^2,\bP)$ constructed in \cite{selinger2012thurston} as well as the extension of $\nu_{\varphi}:\Tbar(S^2,\bP)\to\Hbar(\varphi)$ constructed in \cite{HinichVaintrobAugmented} can be described as follows. Given $[C,\iota,\psi]\in\Tbar(S^2,\bP)$, set $\Gamma$ to be the multicurve on $S^2\setminus \bP$ consisting of curves contracted by $\psi:S^2\to C$, set $\Gamma_2:=\widetilde{\varphi}^{*}(\Gamma)$ to be its pullback (a multicurve on $S^2\setminus\varphi^{-1}(\bP)$), set $X_2$ to be the $\varphi^{-1}(\bP)$-marked topological surface obtained from $S^2$ by contracting to a point each curve in $\Gamma_2$ and set $\psi_2:S^2\to X_2$ to be the quotient map. The branched cover $\varphi$ descends to a branched cover $\underline{\varphi}:X_2\to C$. One can lift the complex structure on $C$ along $\underline{\varphi}$ to obtain an identification $X_2\cong C_2$, where $C_2$ is a $\varphi^{-1}(\bP)$-marked stable curve. We set $\widetilde{\TP}_{\varphi}(C,\iota,\psi)=(C_2, (\psi_2)|_{\varphi^{-1}(\bP)}, \psi_2)\in \Tbar(S^2,\varphi^{-1}(\bP))$. By construction, under the identification $X_2\cong C_2$, the branched cover $\underline{\varphi}:C_2\to C$ is an admissible cover $f$ of profile $\varphi$. We set $\nu_{\varphi}([C,\iota,\psi])=[f]\in \Hbar(\varphi)$. The combinatorial type of $f$ is clearly $\Lambda_{\varphi}(\Gamma)$, and since $[f]\in \Hbar(\varphi)$, $\Lambda_{\varphi}(\Gamma)$ has topological type $\varphi$. Finally we note that any multicurve $\Gamma$ is exactly the contracted multicurve for some $[C,\iota,\psi]\in\Tbar(S^2,\bP)$ -- the lemma follows.  
\end{proof}

\begin{Def}[Definition-Lemma]\label{def:tropicalteichtohurwitz}
    We define a map $\nu_{\varphi}^{\trop}:\ConeCC(S^2,\bP)\to \Htrop(\varphi)$ defined on cones as follows. Given a multicurve $\Gamma$ on $S^2\setminus\bP$, let $\Lambda_{\varphi}(\Gamma)=(\tree_1,\tree_2,F,\edgedeg, \varphi|_{\varphi^{-1}(\bP)}, \legdeg)$ be the combinatorial type of admissible cover as in Definition \ref{def:combinatorialtypefrommulticurve}. Recall that $\tree_1=\tree(\Gamma)$ has an edge $e([\gamma])$ for every $[\gamma]\in\Gamma$. Given $\sum_{[\gamma]\in\Gamma}a_{[\gamma]}[\gamma]\in\cone(\Gamma)\subset\ConeCC(S^2,\bP)$, we obtain a possibly degenerate metrization of $\tree_1$ by assigning length $a_{[\gamma]}$ to $e([\gamma])$, yielding a tropical curve $\Sigma_1$. This tropical curve is the target curve of a unique tropical admissible cover $(\Sigma_1,\Sigma_2,\mathbf{F})\in\cone(\Lambda_{\varphi}(\Gamma))$. We set $\nu_{\varphi}^{\trop}(\sum_{[\gamma]\in\Gamma}a_{[\gamma]}[\gamma])=(\Sigma_1,\Sigma_2,\mathbf{F})$. Note that $\nu_{\varphi}^{\trop}$ is well-defined and piecewise linear, and its restriction to every cone is linear.  
\end{Def}

\begin{lem}\label{lem:nutropsurjective}
    The image of $\nu_{\varphi}^{\trop}:\ConeCC(S^2,\bP)\to \Htropcomb(\varphi)$ is exactly $\Htrop(\varphi)$. 
\end{lem}
\begin{proof}
By \cite{HinichVaintrobAugmented}, $\nu_{\varphi}:\Tbar(S^2,\bP)\to\Hbar(\varphi)$ is surjective, so every combinatorial type of admissible cover $\Lambda$ that has the topological type of $\varphi$ arises as $\Lambda_{\varphi}(\Gamma)$ for some multicurve $\Gamma$. Finally, we note that $\nu_{\varphi}^{\trop}:\cone(\Gamma)\to\cone(\Lambda_{\varphi}(\Gamma))$ is a bijection. 
\end{proof}

\begin{prop}\label{prop:tropicalcorrespondencefrompullback}
    We have 
    \begin{enumerate}
     \item $\rho^{\trop}=\pi_1^{\trop}\circ\nu_{\varphi}^{\trop}$,\label{it:left}
    \item $\widetilde{\rho}^{\trop}\circ\widetilde{\TP}_{\varphi}^{\trop}=\widetilde{\pi}_2^{\trop}\circ\nu_{\varphi}^{\trop}$,\label{it:rightlift}
    \item $\rho^{\trop}\circ\TP_{\varphi}^{\trop}=\pi_2^{\trop}\circ\nu_{\varphi}^{\trop}.$\label{it:right}
    \end{enumerate}
\end{prop}
\begin{rem} Proposition \ref{prop:tropicalcorrespondencefrompullback} says that the tropical moduli space correspondence descends from the tropical Thurston pullback map on weighted multicurves, in the following sense: Parts \eqref{it:left} and \eqref{it:right} of the proposition together imply that Diagram \eqref{diag:tropical} below commutes. This is the third --- tropical --- companion of Diagram \eqref{diag:noncompact} in Section \ref{sec:modulicorrespondence} and its compactified/augmented Diagram \eqref{diag:compact} in Section \ref{sec:compactifiedcorrespondence}. 

\begin{equation}\label{diag:tropical}
  \begin{tikzpicture}
    \matrix(m)[matrix of math nodes,row sep=3em,column sep=8em,minimum
    width=2em] {
      \ConeCC(S^2,\bP) &&\ConeCC(S^2,\bP)\\
      &\Htrop(\varphi)&\\
      \M^{\trop}_{0,\mathbf{P}}&&\M^{\trop}_{0,\mathbf{P}}\\
    };
    \path[-stealth] (m-1-1) edge node [above] {$\TP_{\varphi}^{\trop}$} (m-1-3);
    \path[-stealth] (m-1-1) edge node [above right] {$\nu_{\varphi}^{\trop}$} (m-2-2);
    \path[-stealth] (m-1-1) edge node [left] {$\rho^{\trop}$} (m-3-1);
    \path[-stealth] (m-1-3) edge node [left] {$\rho^{\trop}$} (m-3-3);
    \path[-stealth] (m-2-2) edge node [above left]
    {$\pi_1^{\trop}$} (m-3-1);
    \path[-stealth] (m-2-2) edge node [above right] {$\pi_2^{\trop}$} (m-3-3);
  \end{tikzpicture}
\end{equation}
\end{rem}
\begin{proof}[Proof of Proposition \ref{prop:tropicalcorrespondencefrompullback}]
    The first equality \eqref{it:left} follows directly from the definition (Definition \ref{def:tropicalteichtohurwitz}) of $\nu_{\varphi}^{\trop}$ together with the definition of $\rho^{\trop}$ given in Section \ref{sec:tropicalteichmuller}. For \eqref{it:rightlift}, suppose that $\Gamma$ is any multicurve on $S^2\setminus\bP$, and $\sum_{[\gamma]\in\Gamma}a_{[\gamma]}[\gamma]\in\cone^{\circ}(\Gamma)$. Set $$\Lambda_{\varphi}(\Gamma)=(\tree_1,\tree_2,F,\edgedeg, \varphi|_{\varphi^{-1}(\bP)}, \legdeg)$$ to be the combinatorial type of admissible cover as in Definition \ref{def:combinatorialtypefrommulticurve}. Recall that $\tree_1=\tree(\Gamma)$, so the edges of $\tree_1$ are in bijection with simple closed curves $[\gamma]\in \Gamma$. Also $\tree_2=\tree(\widetilde{\varphi}^{*}(\Gamma))$, so the edges of $\tree_2$ are in bijection with simple closed curves $[\gamma']$ that map to $[\gamma]\in \Gamma$. For any $[\gamma'_0]\in \widetilde{\varphi}^{*}(\Gamma)$ mapping via $\varphi$ to $[\gamma_0]$ with degree $d_{\gamma_0'}$. By \eqref{eq:TropicalThurstonPullback1}, the coefficient of $[\gamma']$ in $\widetilde{\TP}_{\varphi}^{\trop}(\sum_{[\gamma]\in\Gamma}a_{[\gamma]}[\gamma]\in\cone(\Gamma))$ is $\frac{a_{[\gamma_0]}}{d_{\gamma'_0}}$. 
    
    To understand $\widetilde{\pi}_2^{\trop}\circ\nu_{\varphi}^{\trop}(\sum_{[\gamma]\in\Gamma}a_{[\gamma]}[\gamma])$, we observe that first $\nu_{\varphi}^{\trop}(\sum_{[\gamma]\in\Gamma}a_{[\gamma]}[\gamma])$ produces a tropical admissible cover 
$(\Sigma_1,\Sigma_2,\mathbf{F})$ with combinatorial type $\Lambda_{\varphi}(\Gamma)$, then, applying $\widetilde{\pi_2}^{\trop}$, we record the tropical curve $\Sigma_2$, which has underlying tree $\tree_2$. For any $[\gamma_0]\in\Gamma$, the length, on the tropical curve $\Sigma_1$, of the edge $e([\gamma_0])$ of its underlying tree $\tree_1$, is $a_{[\gamma_0]}$. For any $[\gamma'_0]\in \widetilde{\varphi}^{*}(\Gamma)$ mapping via $\varphi$ to $[\gamma_0]$ with degree $d_{\gamma_0'}$, the slope of $\mathbf{F}$ along $e([\gamma'])$ equals $d_{[\gamma']}$, so the length of $e([\gamma'])$ on the tropical curve $\Sigma_2$ equals $\frac{a_{[\gamma_0]}}{d_{\gamma'_0}}$.

To understand $\widetilde{\rho}^{\trop}\circ\widetilde{\TP}_{\varphi}^{\trop}(\sum_{[\gamma]\in\Gamma}a_{[\gamma]}[\gamma])$, we note that $\widetilde{\TP}_{\varphi}^{\trop}(\sum_{[\gamma]\in\Gamma}a_{[\gamma]}[\gamma])$ produces a non-negative weighting on the multicurve $\widetilde{\varphi}^{*}(\Gamma)$, then, applying $\widetilde{\rho}^{\trop}$, we obtain a tropical curve $\Sigma_2'$ with underlying tree $\tree_2=\tree(\widetilde{\varphi}^{*}(\Gamma))$, and edge-lengths corresponding to the weights in $\widetilde{\TP}_{\varphi}^{\trop}(\sum_{[\gamma]\in\Gamma}a_{[\gamma]}[\gamma])$. By \eqref{eq:TropicalThurstonPullback1}, for any $[\gamma'_0]\in \widetilde{\varphi}^{*}(\Gamma)$ mapping via $\varphi$ to $[\gamma_0]$ with degree $d_{\gamma_0'}$, the coefficient of $[\gamma']$ in $\widetilde{\TP}_{\varphi}^{\trop}(\sum_{[\gamma]\in\Gamma}a_{[\gamma]}[\gamma]\in\cone(\Gamma))$ is $\frac{a_{[\gamma_0]}}{d_{\gamma'_0}}$. So the length of $e([\gamma'])$ on the tropical curve $\Sigma_2'$ equals $\frac{a_{[\gamma_0]}}{d_{\gamma'_0}}$.

We conclude that $\Sigma_1=\Sigma_2$, and so $\widetilde{\pi}_2^{\trop}\circ\nu_{\varphi}^{\trop}(\sum_{[\gamma]\in\Gamma}a_{[\gamma]}[\gamma])=\widetilde{\rho}^{\trop}\circ\widetilde{\TP}_{\varphi}^{\trop}(\sum_{[\gamma]\in\Gamma}a_{[\gamma]}[\gamma])$.

For the final equality \eqref{it:right}, observe that since $\pi_2^{\trop}=\mu^{\trop}\circ\widetilde{\pi}_2^{\trop}$, we can write 
\begin{align*}
    \pi_2^{\trop}\circ\nu_{\varphi}^{\trop}&=\mu^{\trop}\circ\widetilde{\pi}_2^{\trop}\circ\nu_{\varphi}^{\trop}\\
    &=\mu^{\trop}\circ\widetilde{\rho}^{\trop}\circ\widetilde{\TP}_{\varphi}^{\trop}\\
    &=\rho^{\trop}\circ\widetilde{\mu}^{\trop}\circ\widetilde{\TP}_{\varphi}^{\trop}\\
    &= \rho^{\trop}\circ\TP^{\trop}_{\varphi}
\end{align*}
The second equality follows from Part \eqref{it:rightlift} of this Lemma. 
\end{proof}

\subsection{Local properties of the tropical moduli space correspondence}
Picking a locally-defined inverse branch of $\pi_1:\cH(\varphi)\to \M_{0,\bP}$ gives us a locally-defined single-valued branch of the multivalued map $\MSC_{\varphi}=(\pi_2)|_{\cH(\varphi)}\circ(\pi_1)|_{cH(\varphi)}$ on $\M_{0,\bP}$. This local single-valued branch has exactly the same local properties as the pullback map $\TP_{\varphi}$ on $\cT(S^2,\bP)$, for example the derivative at a fixed point. There is a tropical analog as follows: $(\pi_1^{\trop})_{\Htrop(\varphi)}$ is finite-to-one and can be thought of as a branched cover. Choosing a cone $\cone(\Lambda)$ of $\Htrop(\varphi)$ yields an inverse branch $(\pi_1^{\trop}|_{\cone(\Lambda)})^{-1}$ of $\pi_1$, defined on $\pi_1(\cone(\Lambda))$, which is a cone of $\M_{0,\bP}^{\trop}$. In fact, if $\Lambda=(\tree_1,\tree_2,\ldots)$, then the image, under $\pi_1^{\trop}(\cone(\Lambda))=\cone(\tree_1)$ and $\pi_2^{\trop}(\cone(\Lambda))\subset \cone(\mu_*(\tree_2))$. Then $$\MSC_{\varphi,\Lambda}^{\trop}:=(\pi_2)|_{\cone(\Lambda)}\circ (\pi_1^{\trop}|_{\cone(\Lambda)})^{-1}:\cone(\tree_1)\to\cone(\mu_*(\tree_2))$$ is the single-valued linear \textit{branch associated to $\Lambda$} of the tropical moduli space correspondence $\MSC_{\varphi}^{\trop}$. If $\cone(\Lambda)$ is a maximal cone of $\Htrop(\varphi)$, then we will refer to $\MSC_{\varphi,\Lambda}^{\trop}$ as a \textit{local} branch of $\MSC_{\varphi}^{\trop}$. Proposition \ref{prop:localtropicalcorrespondence} says that every branch of $\MSC_{\varphi}^{\trop}$ is given by some Thurston linear transformation (Section \ref{sec:tropicalpullback} Equation \eqref{eq:TropicalThurstonPullback1}). 

\begin{prop}\label{prop:localtropicalcorrespondence}
    Suppose $\cone(\Lambda)$ is any cone of $\Htrop(\varphi)$, that $\Lambda=\Lambda_{\varphi}(\Gamma)=(\tree_1,\tree_2,\ldots)$, and that $\Gamma$ is any multicurve on $S^2\setminus \bP$ such that $\Lambda=\Lambda_{\varphi}(\Gamma)$. Let $M_1$ be the matrix of the branch $\MSC_{\varphi,\Lambda_{\varphi}(\Gamma)}^{\trop}$, written with respect to the natural generators $\edges(\tree_1)$ of $\cone(\tree_1)$ and $\edges(\mu_*(\tree_2))$ of $\cone(\mu_*(\tree_2))$. Let $M_2$ be the matrix of Thurston linear transformation $\TLT_{\varphi,\Gamma}$(\eqref{eq:TropicalThurstonPullback1}), written with respect to the natural bases of $\R^{\Gamma}$ and $\R^{\varphi^*(\Gamma)}$. Then $M_1=M_2$. 
\end{prop}

\begin{proof}
    Write $\Gamma=\{[\gamma_1],\ldots,[\gamma_r]\}$, which means that $\tree_1=\tree(\Gamma)$ has edges $e([\gamma_1],\ldots, e([\gamma_r])$.  Write $\varphi^{*}(\Gamma)=\{[\gamma'_1],\ldots, [\gamma'_s]\}$, which means that $\mu_*(\tree_2)=\tree(\varphi^{*}(\Gamma))$ has edges $e([\gamma'_1]),\ldots, e([\gamma'_s])$. Recall that 
    \begin{itemize}
    \item $\cone(\Gamma)\subset \ConeCC(S^2,\bP)$ is generated by $[\gamma_1],\ldots,[\gamma_r]$,
    \item  $\cone(\tree_1)$ is generated by $e([\gamma_1],\ldots, e([\gamma_r])$,
    \item $(\rho^{\trop})|_{\cone(\Gamma)}:\cone(\Gamma)\to \cone(\tree_1)$ is an isomorphism taking $[\gamma_i]$ to $e([\gamma_i])$,
    \item $\cone(\varphi^{*}(\Gamma))\subset \ConeCC(S^2,\bP)$ is generated by $[\gamma'_1],\ldots,[\gamma'_s]$,
    \item  $\cone(\mu_*(\tree_2))$ is generated by $e([\gamma'_1],\ldots, e([\gamma'_s])$,
    \item $(\rho^{\trop})|_{\cone(\varphi^*(\Gamma))}:\cone(\varphi^*(\Gamma))\to \cone(\mu_*(\tree_2))$ is an isomorphism taking $[\gamma'_j]$ to $e([\gamma'_j])$,
    \item  The matrix of $(\TP_{\varphi}^{\trop})|_{\cone(\Gamma)}:\cone(\Gamma)\to\cone(\varphi^*(\Gamma))$, with respect to the generators $[\gamma_1],\ldots,[\gamma_r]$ and $[\gamma'_1],\ldots,[\gamma'_s]$ equals $M_2$. 
\end{itemize}

By Proposition \ref{prop:tropicalcorrespondencefrompullback} \eqref{it:right}, 
\begin{align*}
    (\rho^{\trop})_{\cone(\Gamma)}=(\pi_1^{\trop})|_{\cone(\Lambda_{\varphi}(\Gamma))}\circ(\nu_{\varphi}^{\trop})|_{\cone(\Gamma)}
\end{align*}

    By Proposition \ref{prop:tropicalcorrespondencefrompullback} \eqref{it:left}, 
\begin{align*}
    (\rho^{\trop})|_{\cone(\varphi^*(\Gamma))}\circ(\TP_{\varphi}^{\trop})|_{\cone(\Gamma)}=(\pi_2^{\trop})|_{\cone(\Lambda_{\varphi}(\Gamma))}\circ\nu_{\varphi}^{\trop})|_{\cone(\Gamma)}
\end{align*}

So, as maps from $\cone(\tree_1)$ to $\cone(\mu_*(\tree_2))$:
\begin{align}
    &(\rho^{\trop})|_{\cone(\varphi^*(\Gamma))}\circ(\TP_{\varphi}^{\trop})|_{\cone(\Gamma)}\circ((\rho^{\trop})|_{\cone(\Gamma)})^{-1}\label{it:firstline}\\ =&(\pi_2^{\trop})|_{\cone(\Lambda_{\varphi}(\Gamma))}\circ\nu_{\varphi}^{\trop})|_{\cone(\Gamma)}\circ((\rho^{\trop})|_{\cone(\Gamma)})^{-1}\notag\\
    =&(\pi_2^{\trop})|_{\cone(\Lambda_{\varphi}(\Gamma))}\circ\nu_{\varphi}^{\trop})|_{\cone(\Gamma)}\circ((\nu_{\varphi}^{\trop})|_{\cone(\Gamma)})^{-1}\circ((\pi_1^{\trop})|_{\cone(\Lambda_{\varphi}(\Gamma))})^{-1}\notag\\
     =&(\pi_2^{\trop})|_{\cone(\Lambda_{\varphi}(\Gamma))}\circ((\pi_1^{\trop})|_{\cone(\Lambda_{\varphi}(\Gamma))})^{-1}\label{it:lastline}
\end{align}

In terms of the generators $\{e([\gamma_i])\}$ for $\cone(\tree_1)$ and $\{e([\gamma_j])\}$ for $\cone(\mu_*(\tree_2))$, it is clear that \eqref{it:firstline} has matrix $M_2$ and (by definition) \eqref{it:lastline} has matrix $M_1$. The proposition follows.

\end{proof}


\subsection{Weakly fixed cones} A \textit{weakly fixed cone} of $\MSC_{\varphi}^{\trop}$ is a cone $\cone(\Lambda)$ of $\Htrop(\varphi)$ such that $\pi_2^{\trop}(\cone(\Lambda))\subset\pi_1^{\trop}(\cone(\Lambda))$. If $\Lambda=(\tree_1,\tree_2,\ldots)$ then $\cone(\Lambda)$ is weakly fixed if and only if $\mu_*(\tree_2)$ is obtained from $\tree_1$ by edge-contraction (possibly of the empty set of edges). In this case, via the inclusion of cones $\cone(\mu_*(\tree_2))\into\cone(\tree_1)$, the branch $\MSC_{\varphi,\Lambda}^{\trop}$ of $\MSC_{\varphi}^{\trop}$ can be considered to be a self-map of $\cone(\tree_1)$. By the Perron-Frobenius theorem on non-negative matrices: 
\begin{itemize}
    \item The dominant eigenvalue $\lambdadom(\Lambda)$ of $\MSC_{\varphi,\Lambda}^{\trop}$ is real and non-negative, and has an eigenvector in $\cone(\tree_1)$, and
    \item if $\MSC_{\varphi,\Lambda}^{\trop}$ has an eigenvector in $\cone^{\circ}(\tree_1)$, then its eigenvalue is $\lambdadom(\Lambda)$. 
\end{itemize}

\begin{prop}\label{prop:weaklyfixedcone}\hfill
\begin{enumerate}
    \item If $\Gamma$ is a $\varphi$-stable multicurve then $\cone(\Lambda_{\varphi}(\Gamma))$ is a weakly fixed cone of $\MSC_{\varphi}^{\trop}$, and the dominant eigenvalue $\lambdadom(\Lambda_{\varphi}(\Gamma))$ equals the Thurston eigenvalue $\lambdathurst(\Gamma)$ associated to $\Gamma$. \label{it:invconefromgamma}
    \item If $\cone(\Lambda)$ is weakly fixed cone of $\MSC_{\varphi}^{\trop}$ then there exists $\varphi'$ Hurwitz equivalent to $\varphi$ together with a $\varphi'$-stable multicurve $\Gamma$ such that $\Lambda=\Lambda_{\varphi'}(\Gamma)$. \label{it:invgammafrominvcone}
\end{enumerate}
\end{prop}
\begin{proof}

   For \eqref{it:invconefromgamma}, write $\Lambda_{\varphi}(\Gamma)=(\tree_1,\tree_2,\ldots)$. If $\Gamma$ is $\varphi$-stable then $\varphi^{*}(\Gamma)\subset(\Gamma)$, which implies that $\mu_*(\tree_2)=\tree(\varphi^{*}(\Gamma))$ is obtained from $\tree_1=\tree(\Gamma)$ by edge-contraction, which in turn implies that $\Lambda_{\varphi}(\Gamma)$ is weakly fixed by $\MSC_{\varphi}^{\trop}$. By Proposition \ref{prop:localtropicalcorrespondence}, the branch $\MSC_{\varphi,\Lambda}^{\trop}$ and the Thurston linear transformation $\TLT_{\varphi,\Gamma}$ have the same matrix representative therefore have the same dominant eigenvalue. So $\lambdadom(\Lambda_{\varphi}(\Gamma))=\lambdathurst(\Gamma)$. 

   Now, suppose that $\cone(\Lambda)$ is a weakly fixed come of $\MSC_{\varphi}^{\trop}$.  Pick any multicurve $\Gamma$ such that $\Lambda=\Lambda_{\varphi}(\Gamma)$.   
   Now, suppose that $\cone(\Lambda)$ is a weakly fixed come of $\MSC_{\varphi}^{\trop}$ --- the fact that $\cone(\Lambda)$ is weakly fixed implied that $\mu_*(\tree_2)=\tree_1$.  Pick any multicurve $\Gamma$ such that $\Lambda=\Lambda_{\varphi}(\Gamma)$; for such $\Gamma$, $\tree(\varphi^{*}(\Gamma))=\tree(\Gamma)$, which means that there is a homeomorphism $\alpha:S^2\to S^2$, fixing $\bP$ pointwise, whose mapping class $[\alpha]\in\Mod(S^2,\bP)$ satisfies $\alpha_{*}(\Gamma)=\varphi^*(\Gamma)$. Set $\varphi'=\varphi\circ\alpha$. It is clear that $\varphi'$ and $\varphi$ are Hurwitz equivalent. Also, $(\varphi')^{*}(\Gamma)= \alpha^*(\varphi^*(\Gamma))=\Gamma$, so $\Gamma$ is $\varphi'$-stable. Finally, by construction, $\Lambda_{\varphi'}(\Gamma)=\Lambda_{\varphi}(\Gamma)=\Lambda$, proving \eqref{it:invgammafrominvcone}.  
\end{proof}

\subsection{Weakly fixed rays}

\begin{Def}\label{def:ray}
  A \textit{ray} of $\M_{0,\bP}^{\trop}$ is a subset of the form $\R_{\ge0}\cdot\Sigma$, where $\Sigma$ is not the one-vertex tree, or, equivalently, $\Sigma$ is not the cone point of $\M_{0,\bP}^{\trop}$). A ray of $\Htrop(\varphi)$ (resp. of $\Htropcomb(\varphi)$) is a subset of the form $\R_{\ge0}\cdot (\Sigma_1,\Sigma_2,\mathbf{F})$, where $\Sigma_1$ is not the one-vertex tree, or, equivalently, $(\Sigma_1,\Sigma_2,\mathbf{F})$ is not the cone point of $\Htrop(\varphi)$ (resp. of $\Htropcomb(\varphi)$).
\end{Def}

The image, under $\pi_1^{\trop}$, of a ray of $\Htrop(\varphi)$ is a ray of $\M_{0,\bP}^{\trop}$. The image, under $\pi_2^{\trop}$, of a ray of $\Htrop(\varphi)$ is either ray of $\M_{0,\bP}^{\trop}$ or is the cone point of $\M_{0,\bP}^{\trop}$.
 
\begin{Def}\label{def:weaklyfixedray}
A \textit{weakly fixed ray} of the tropical moduli space correspondence is a ray of $\Htrop(\varphi)$ whose image under $\pi_2^{\trop}$ is contained in its image under $\pi_1^{\trop}$. A weakly fixed ray $R$ of the tropical moduli space correspondence has associated to it a \textit{scaling factor} $\lambdascal(R)\in \R_{\ge0}$; this is defined to be the factor such that the map $\pi_2^{\trop}\circ(\pi_1^{\trop})^{-1}:\pi_1(R)\to \pi_1(R)$ scales by $\lambdascal(R)$. 
\end{Def}

Suppose that $R=\R_{\ge0}\cdot (\Sigma_1,\Sigma_2,\mathbf{F})$ is a weakly fixed ray of the moduli space correspondence. Set $\Lambda=(\tree_1,\tree_2,\ldots)$ to be the combinatorial type of $(\Sigma_1,\Sigma_2,\mathbf{F})$, so $\cone(\Lambda)$ is the smallest cone containing $R$. Then $\tree_1$ and $\tree_2$ are the underlying $\bP$- and $\varphi^{-1}(\bP)$- marked trees of $\Sigma_1$ and $\Sigma_2$ respectively. The tropical curve $\mu^{\trop}(\Sigma_2)$ is obtained from the tropical curve $\Sigma_1$ by scaling all the edges by $\lambdascal(R)$. Either
\begin{enumerate}
    \item $\lambdascal(R)=0$, and $\mu_{*}(\tree_2)$ is the one-vertex tree, or
    \item $\lambdascal(R)\neq 0$, and $\mu_{*}(\tree_2)=\tree_1$.
\end{enumerate}

In either case, we conclude that $\cone(\Lambda)$ is weakly fixed by $\MSC_{\varphi}^{\trop}$, and that $(\Sigma_1,\Sigma_2,\mathbf{F})$ is a strictly positive eigenvector for $\MSC_{\varphi,\Lambda}^{\trop}$. By the Perron-Frobenius theorem for non-negative matrices, $\lambdascal(R)=\lambdadom(\Lambda)$. 

Conversely, for $\Lambda=(\tree_1,\tree_2,\ldots)$, if $\cone(\Lambda)$ is weakly fixed but not the cone point of $\Htrop(\varphi)$, then we set $\mathrm{Eigenspace}_{+}(\Lambda)\subset\cone(\Lambda)$ to be the intersection with $\cone(\Lambda)$ of the dominant eigenspace of $\MSC_{\varphi,\Lambda}^{\trop}$. Then $\mathrm{Eigenspace}_{+}(\Lambda)\subset\cone(\Lambda)$ is a union of weakly fixed rays, all of whose scaling factors equal $\lambdadom(\Lambda)$. 

If $\Gamma$ is a $\varphi$-stable multicurve then we set  
$$\mathrm{FixedRays}_{\varphi}(\Gamma):=\{R \mbox{ ray }\st R\subset \mathrm{Eigenspace}_{+}(\Lambda(\Gamma))\}.$$

\begin{prop}\label{prop:weaklyfixedray}\hfill
\begin{enumerate}
    \item If $\Gamma$ is a non-empty $\varphi$-stable multicurve then $\mathrm{FixedRays}_{\varphi}(\Gamma)$ is non-empty, and for all $R\in\mathrm{FixedRays}_{\varphi}(\Gamma)$, $\lambdascal(R)=\lambdathurst(\Gamma).$
    
    \item If $R$ is a weakly fixed ray of the moduli space correspondence then there exist $\varphi'$ Hurwitz equivalent to $\varphi$ and $\varphi'$-stable multicurve $\Gamma$ such that $R\in \mathrm{FixedRays}_{\varphi'}(\Gamma)$.
\end{enumerate}
\end{prop}
\begin{proof}
    This follows directly from Proposition \ref{prop:weaklyfixedcone}.
\end{proof}

\begin{cor}\label{cor:obstructionsandrays}\hfill
\begin{enumerate}
    \item If $\Gamma$ is a Thurston obstruction for $\varphi$ then $\mathrm{FixedRays}_{\varphi}(\Gamma)$ is non-empty, and for all $R\in\mathrm{FixedRays}_{\varphi}(\Gamma)$, $\lambdascal(R)\ge 1.$
    
    \item If $R$ is a weakly fixed ray of the moduli space correspondence with $\lambdascal(R)\ge 1$ then there exist $\varphi'$ Hurwitz equivalent to $\varphi$ and a Thurston obstruction $\Gamma$ for $\varphi$ such that $R\in \mathrm{FixedRays}_{\varphi'}(\Gamma)$.
\end{enumerate}
\end{cor}

\begin{rem}
    Suppose that $\R_{\ge0}\cdot(\Sigma_1,\Sigma_2,\mathbf{F})$ is a weakly-fixed ray of $\MSC^{\trop}_{\varphi}$ with scaling factor $\lambdascal(R)>0$. The $\bP$-marked tropical curves $\mu^{\trop}(\Sigma_2)$ and $\Sigma_1$ have the same underlying tree, and the length of an edge on $\mu^{\trop}(\Sigma_2)$ is $\lambdascal(R)$ times the length of the corresponding edge on $\Sigma_1$. There is a ``scale by $\lambdascal(R)$" map $h:\Sigma_1\to \mu^{\trop}(\Sigma_2)$. The composite $\mathbf{F}\circ h:\Sigma_1\to\Sigma_1$ is a piecewise linear self-map whose dynamics could be of interest. This idea has been mentioned by Shishikura \cite{ShishikuraTalk} in the context of $\varphi$-stable multicurves. 
\end{rem}

\begin{rem}\label{rem:notdefinedonlink} The link $\Delta_{\bP}$ of $\M_{0,\bP}^{\trop}$ (see Remark \ref{rem:link}) has the structure of a compact polyhedral complex. Given that $\MSC_{\varphi}^{\trop}$ is a piecewise linear multivalued self-map of $\M_{0,\bP}^{\trop}$, we might hope that it descends to a piecewise affine multivalued self-map of $\Delta_{\bP}$. Unfortunately, this is not the case. If we try to define such an induced multivalued self-map $\overline{\MSC}_{\varphi}^{\trop}$ of $\Delta_{\bP}$, we hit a snag as follows. Given a ray $R\subset \M_{0,\bP}^{\trop}$ and a local single-valued branch $\MSC_{\varphi,\Lambda}^{\trop}$ of $\MSC_{\varphi}^{\trop}$ whose domain includes $R$, we would like to be able to specify the image of $R\in\Delta_{\bP}$ under a corresponding single-valued branch $\overline{\MSC}_{\varphi,\Lambda}^{\trop}$ of the desired multivalued self-map $\overline{\MSC}_{\varphi}^{\trop}$ of $\Delta_{\bP}$. The image, under $\MSC_{\varphi, \Lambda}^{\trop}$, of $R$ is either another ray $R'$, or the cone point of $\M_{0,\bP}^{\trop}$. In the former case, we could set $\overline{\MSC}_{\varphi, \Lambda}^{\trop}(R)=R'\in\Delta_{\bP}$. However, in the latter case, there is no point of $\Delta_{\bP}$ corresponding to the cone point of $\M_{0,\bP}^{\trop}$, so there is no good value for $\overline{\MSC}_{\varphi, \Lambda}^{\trop}(R)$. One could possibly think of $\overline{\MSC}_{\varphi}^{\trop}$ as a multivalued ``map with indeterminacy", analogous to a meromorphic map. Note that if $R\subset\M_{0,\bP}^{\trop}$ is a weakly fixed ray of $\MSC_{\varphi}^{\trop}$ with strictly positive scaling factor, then $R\in\Delta_{\bP}$ is a weakly fixed point of $\overline{\MSC}_{\varphi}^{\trop}$. Making this strategy precise, and studying the global dynamical properties of an induced action on the link are currently the subject of investigation by the author and C. Favre.     
\end{rem}

\bibliographystyle{alpha} 
\bibliography{BigRefs.bib}
\end{document}